\documentclass[journal]{IEEEtran}

\ifCLASSINFOpdf

\else

\fi

\usepackage{amsthm}
\usepackage{epsfig} 
\usepackage{epstopdf}
\usepackage{amsmath}
\usepackage{amsthm}
\usepackage{comment}
\usepackage{url}
\usepackage{algorithm}
\usepackage{algorithmic}

\usepackage{amssymb}
\usepackage{amsmath}
\usepackage{enumerate}
\usepackage{autobreak}

\usepackage{amsfonts}
\usepackage{bm}
\usepackage{color}
\usepackage{cite}
\usepackage{stfloats}
\usepackage{chemarrow}
\usepackage{extarrows}
\allowdisplaybreaks

\newboolean{showcomments}
\setboolean{showcomments}{true}
\newcommand{\zjwang}[1]{\ifthenelse{\boolean{showcomments}}
	{ \textcolor[rgb]{1,0,1}{(ZW:  #1)}}{}}
\newcommand{\fliu}[1]{\ifthenelse{\boolean{showcomments}}
	{ \textcolor{blue}{(FL:  #1)}}{}}

\theoremstyle{definition}
\newtheorem{theorem}{Theorem}
\newtheorem{lemma}[theorem]{Lemma}
\newtheorem{corollary}[theorem]{Corollary}

\theoremstyle{definition}

\newtheorem{remark}{Remark}

\title{Convergence Analysis of Dual Decomposition Algorithm in Distributed Optimization: Asynchrony and Inexactness}

\begin{document}
\graphicspath{{Figures/}}

\author{
	Yifan Su, \IEEEmembership{Student Member, IEEE}, 
	Zhaojian Wang, \IEEEmembership{Member, IEEE}, 
	Ming Cao, \IEEEmembership{Senior Member, IEEE}, 
	
	Mengshuo Jia, \IEEEmembership{Student Member, IEEE}, 
	Feng Liu, \IEEEmembership{Senior Member, IEEE}
	\thanks{This work was supported by the Joint Research Fund in Smart Grid (No.U1966601) under cooperative agreement between the National Natural Science Foundation of China (NSFC) and State Grid Corporation of China. \textit{(Corresponding author: Feng Liu)}}
	\thanks{Yifan Su, Zhaojian Wang, Mengshuo Jia, and Feng Liu are with the State Key Laboratory of Power System, Department of Electrical Engineering, Tsinghua University, Beijing 100084, China (e-mail: suyf19@mails.tsinghua.edu.cn; wangzhaojiantj@163.com; jms16@mails.tsinghua.edu.cn; lfeng@tsinghua.edu.cn).}
	\thanks{Zhaojian Wang is also with the Key Laboratory of System Control, and Information Processing, Ministry of Education of China, Department of Automation, Shanghai Jiao Tong University, Shanghai 200240, China.}
	\thanks{Ming Cao is with the Faculty of Science, and Engineering, University of Groningen, Groningen 9747 AG, The Netherlands (e-mail: ming.cao@ieee.org).}
}

\markboth{Journal of \LaTeX\ Class Files,~Vol.~xx, No.~xx, xx~xxxx}%
{Shell \MakeLowercase{\textit{et al.}}: Bare Demo of I EEEtran.cls for IEEE Journals}

\maketitle

\begin{abstract}
Dual decomposition is widely utilized in distributed optimization of multi-agent systems. In practice, the dual decomposition algorithm is desired to admit an asynchronous implementation due to imperfect communication, such as time delay and packet drop. In addition, computational errors also exist when individual agents solve their own subproblems. In this paper, we analyze the convergence of the dual decomposition algorithm in distributed optimization when both the asynchrony in communication and the inexactness in solving subproblems exist.
We find that the interaction between asynchrony and inexactness slows down the convergence rate from $\mathcal{O} ( 1 / k )$ to $\mathcal{O} ( 1 / \sqrt{k} )$.
Specifically, with a constant step size, the value of objective function converges to a neighborhood of the optimal value, and the solution converges to a neighborhood of the exact optimal solution. Moreover, the violation of the constraints diminishes  in $\mathcal{O} ( 1 / \sqrt{k} )$.  
Our result generalizes and unifies the existing ones that only consider either asynchrony or inexactness. Finally, numerical simulations validate the theoretical results.

\end{abstract}

\begin{IEEEkeywords}
Dual decomposition, distributed optimization, asynchronous algorithm, inexact algorithm, multi-agent system
\end{IEEEkeywords}


\section{Introduction}

\subsection{Background}
Dual decomposition is widely utilized in solving distributed optimization problems of multi-agent systems (MASs), such as communication networks \cite{papachristodoulou2010delay, he2011cross, magnusson2017convergence}, computer vision \cite{komodakis2007mrf, strandmark2010parallel}, and  power systems \cite{alkano2017asynchronous, huang2018stochastic, falsone2019decentralized}.
A dual decomposition algorithm usually involves two phases in each iteration: 1) a  coordinator updates the dual variables (Lagrangian multipliers) and 2) individual agents solve their subproblems locally \cite{Chiang1, Chiang2, Chiang3}. Then the dual variables and the solutions to subproblems are exchanged between the coordinator and the agents via the communication network to execute the next iteration. 
During the iterative process,  the asynchrony in communication and the inexactness in solving subproblems may undermine the convergence of the algorithm. In the literature, these two issues are addressed separately even though they always co-exist. In this regard, this paper analyzes the convergence of the dual decomposition based distributed optimization (DD-DO) algorithm considering asynchrony and inexactness simultaneously.

\subsection{Related Works}
Dual decomposition is commonly regarded as a first-order (sub)gradient ascent method with respect to the dual problem. Under ideal conditions, the convergence of dual decomposition has been thoroughly studied. For a diminishing step size $\alpha_k$ satisfying $\sum_{k=0}^\infty \alpha_k \to \infty, \sum_{k=0}^\infty \alpha_k^2 < \infty$, the gradient and subgradient algorithms converge to the optimal values \cite[Prop. 8.2.4]{bertsekas2003convex}. For a constant step size, the gradient algorithm still converges to the optimal value \cite[Prop. 1.2.3]{bertsekas1997nonlinear}, while the subgradient method converges to a neighborhood of the optimal value \cite[Prop. 8.2.2]{bertsekas2003convex}. However, if the asynchrony in communication and the inexactness in solving subproblems are considered, the convergence of dual decomposition requires further studies. Next, we give a short review from these two aspects.



 
\subsubsection{Asynchrony in communication} In practice, the implementation of the DD-DO algorithm usually suffers from asynchrony due to packet drop, time delay in communications, and non-identical computational rates, etc. In this situation, the synchronous DD-DO algorithm may cause longer idle time since the coordinator and the agents have to wait for the latest information from their neighbors in order to execute the next iteration\cite{wang2020asynchronous}. To circumvent this problem, an asynchronous DD-DO algorithm is proposed in \cite{bolognani2014distributed} by updating the dual variables and solving subproblems immediately using the previously stored information, if the latest information happens to be unavailable. In \cite{Steven}, the convergence of the asynchronous DD-DO algorithm is studied, showing that the algorithm converges to the optimal solution under a bounded time delay. In \cite{magnusson2020distributed}, an asynchronous distributed voltage control algorithm based on dual decomposition is formulated and solved. In \cite{lee2015convergence}, a distributed quadratic programming method based on asynchronous dual decomposition is proved to converge with a given stationary probability of asynchrony. Reference \cite{notarnicola2017distributed} proposes an asynchronous partitioned dual decomposition algorithm for fully distributed optimization over peer-to-peer networks, where the algorithm converges in probability with the independent and identically distributed (i.i.d.) delays.

\subsubsection{Inexactness in solving subproblems} In the dual decomposition algorithm, the solutions to subproblems of individual agents will inevitably deviate from the exact optimal solutions, depending on the preset error tolerances of solvers, the types of problems, and the accuracy of parameters. The inexactness issue may lead to considerable error or even divergence of the algorithm due to the accumulation of subproblem errors during iteration. To address the issue, the averaging scheme is suggested in recent years by taking the average of decision variables over the iteration horizon. Reference \cite{devolder2014first} utilizes the inexact oracle to study the dual decomposition algorithm. In \cite{necoara2013rate}, the inexact dual decomposition is proved to have the convergence rates of $\mathcal{O} ( 1 / k )$. An inexact DD-DO algorithm to solve a Laplacian consensus problem is studied in \cite{fazlyab2018distributed}, where the deviation of solution diminishes exponentially considering the exponentially decayed error. In \cite{zhang2020augmented}, the iteration complexity of the inexact augmented Lagrangian method for constrained convex programming is studied, where the convergence rate is $\mathcal{O} ( 1 / k )$ even with the nonsmooth objective function. Reference \cite{mehyar2004optimization} analyses the convergence of dual decomposition with inexact updating of dual variables, where the choice of step size is presented to help the algorithm enter an attraction region in finite steps.

\subsection{Contributions}
It should be noted that the works mentioned above only consider either the asynchrony \cite{bolognani2014distributed, magnusson2020distributed, Steven, lee2015convergence, notarnicola2017distributed} or the inexactness \cite{devolder2014first, necoara2013rate, fazlyab2018distributed, zhang2020augmented, mehyar2004optimization} in dual decomposition separately, although they always co-exist in distributed optimization. In this paper, we analyze the convergence of the DD-DO algorithm considering asynchrony and inexactness simultaneously. Specifically, under mild conditions, we prove the  convergence of the asynchronous and inexact DD-DO algorithm, whose characteristics include:
\begin{enumerate}
\item \textbf{Sublinear convergence rate.} Under ideal conditions, the DD-DO algorithm converges in $\mathcal{O} ( 1 / k )$ from \cite{beck20141, magnusson2018communication}. We prove that the interaction of asynchrony and inexactness slows down the convergence rate to $\mathcal{O} ( 1 / \sqrt{k} )$.  We also show that a \textit{constant} step size is enough to obtain the above convergence performance, which is more applicable in practice than using diminishing step sizes as in \cite{magnusson2017convergence, bertsekas2003convex}.

\item \textbf{Suboptimality and Feasibility.} We show that the value of primal variable converges to a neighborhood of the optimal solution to the primal problem, while the value of primal (dual) objective converges to a neighborhood of the optimal value of the primal (dual) problem, both in $\mathcal{O} ( 1 / \sqrt{k} )$. We also give upper bounds of these neighborhoods, which are positively correlated to the degrees of asynchrony and inexactness.
Moreover, the violation of constraints diminishes in the rate of $\mathcal{O} ( 1 / \sqrt{k} )$, even though the solutions to subproblems are inexact in each iteration.

\item \textbf{Generality.} Our convergence results generalize and unify the existing works that only consider asynchrony \cite{Steven} or inexactness \cite{necoara2013rate}.  By simply setting the inexactness or asynchrony parameter as zero, our result reduces to that given in \cite{Steven} or \cite{necoara2013rate}, respectively.   Our work also first gives the $\mathcal{O}(1/k)$ convergence rate for the \textit{asynchronous} DD-DO algorithm, which, to the best of our knowledge, has not been presented in the existing literature \cite{alkano2017asynchronous, magnusson2020distributed, Steven, lee2015convergence}.

\end{enumerate}

\subsection{Organization}
The rest of this paper is organized as follows. Section II formulates the optimal MAS operation problem, and solves it by the synchronous and exact DD-DO algorithm. In Section III, the asynchrony and inexactness are formulated and analyzed in the DD-DO algorithm. Section IV proves the convergence of the asynchronous and inexact algorithm. Section V gives numerical results and Section VI concludes this paper.

\textit{Notations:} In this paper, we use $\mathbb{R}^n$ ($\mathbb{R}^n_+$) to denote the $n$-dimensional (nonnegative) Euclidean space.
For a column vector $\bm{x}\in \mathbb{R}^n$ (matrix $A\in\mathbb{R}^{m\times n}$), $\bm{x}^T$ ($A^T$) denotes its transpose.
For $\bm{x}, \bm{y} \in \mathbb{R}^n$, we denote the inner product by $\left< \bm{x}, \bm{y} \right> = \bm{x}^T \bm{y}$, and the 2-norm by $\left\| \bm{x} \right\| = \sqrt{\left< \bm{x}, \bm{x} \right>}$. 
For a vector $\bm{x} \in \mathbb{R}^n$, $x_i$ stands for the $i$th entry. $col\left\{ \bm{x}_i \right\}_{i\in\mathcal{I}}$ stacks the vectors $\bm{x}_i$ as a new column vector in the order of the index set $\mathcal{I}$. 
For a matrix $A \in \mathbb{R}^{m\times n}$, $\left\| A \right\|$ and $\left\| A \right\|_F$ stand for the 2-norm and Frobenius norm, respectively. Note that $\left\| A \right\| \le \left\| A \right\|_F$.
For a set $\Omega$, $\left| \Omega \right|$ stands for its cardinality. For a closed convex set $\Omega \subset \mathbb{R}^n$, we define the projection of $\bm{x} \in \mathbb{R}^n$ onto $\Omega$ as $\left[ \bm{x} \right]_\Omega = \arg \min_{y\in\Omega} \left\| \bm{y} - \bm{x} \right\|$. Specially, denote by $\left[\bm{x}\right]^+$  the projection onto $\mathbb{R}^n_+$. This projection is nonexpansive, i.e., $\left\| \left[ \bm{x} \right]_\Omega - \left[ \bm{y} \right]_\Omega \right\| \le \left\| \bm{x} - \bm{y} \right\|, ~ \forall \bm{x},\bm{y} \in \mathbb{R}^n$.
	

\section{DD-DO Algorithm in MASs}
In this section, we formulate the operation problem of the MAS, and solve it by the conventional DD-DO algorithm.

\subsection{Primal Problem}
We focus on the large-scale MAS with a set of agents denoted by $\mathcal{N}$. Each agent $i\in \mathcal{N}$ can make its decision $\bm{x}_i\in\mathbb{R}^{n_i}$ in a local feasible region $\mathcal{X}_i$, and meanwhile causes a cost $f_i\left(\bm{x}_i\right)$. Our objective is to minimize the aggregate cost with restrictions on global constraints and local feasible regions, i.e., solve the following optimization problem, called the primal problem:
\begin{subequations}\label{primal_problem}
\begin{align}
\min_{\bm{x}} ~ & F\left(\bm{x}\right) = \sum_{i\in\mathcal{N}} f_i\left(\bm{x}_i\right) \\
\text{s.t.} ~ & \bm{x}_i \in \mathcal{X}_i, ~ \forall i \in \mathcal{N} \label{primal_problem_cons1} \\
& A \bm{x} \le \bm{b}  \label{primal_problem_cons2}
\end{align}
\end{subequations}
where $\bm{x} = col\left\{ \bm{x}_i \right\}_{i\in\mathcal{N}} \in \mathbb{R}^n$ is the aggregate decision vector; $\mathcal{X} = \Pi_{i\in \mathcal{N}} \mathcal{X}_i$ is the aggregate feasible region; \eqref{primal_problem_cons1} represents local feasible regions of agents; \eqref{primal_problem_cons2} is the global constraints. Matrix $A \in \mathbb{R}^{m\times n}$ and vector $\bm{b} \in \mathbb{R}^m$ are constants. Let $A_i \in \mathbb{R}^{m \times n_i}$ denote the $i$th sliced block of $A = \left( A_1, A_2, ... , A_{\left| \mathcal{N} \right| } \right)$. Then \eqref{primal_problem_cons2} can be replaced by
\begin{align*}
\sum_{i\in \mathcal{N}} A_i \bm{x}_i \le \bm{b}
\end{align*}

Throughout the paper, we make the following assumptions on the primal problem.

\textit{Assumption A1:}

\begin{enumerate}
\item The cost function $f_i \left(\cdot\right)$ is $c_i$-strongly convex and twice differentiable over $\mathcal{X}_i$. Hence the objective function $F\left(\cdot\right)$ is strongly convex with $c_F = \min_{i\in\mathcal{N}} c_i$ and twice differentiable over $\mathcal{X}$.

\item The feasible region $\mathcal{X}_i$ is a nonempty, compact, and convex set. Hence $\mathcal{X}$ is also nonempty, compact, and convex.

\item There exists a strictly feasible interior point in $\mathcal{X}$ such that \eqref{primal_problem_cons2} holds.
\end{enumerate}

Under Assumption A1, problem \eqref{primal_problem}  enjoys a unique primal optimal solution denoted by $\bm{x}^*\in \mathcal{X}$. Let $F^* = F\left( \bm{x}^* \right)$ denote the optimal value of \eqref{primal_problem}.

\subsection{Dual Decomposition}
Considering \eqref{primal_problem}, define the Lagrangian
\begin{align*}
\mathcal{L}\left( \bm{x}; \bm{\lambda}\right) = F\left(\bm{x}\right) + \left< \bm{\lambda}, A \bm{x} - \bm{b} \right>
\end{align*}
where $\bm{\lambda} \in \mathbb{R}^m_+$ is the Lagrangian multiplier of \eqref{primal_problem_cons2}. Thereafter, we call $\bm{x}$ and $\bm{\lambda}$ the primal and dual variables, respectively.

Hence the dual problem of \eqref{primal_problem} is given by 
\begin{align} \label{dual_problem}
\max_{\bm{\lambda} \ge 0} ~ \min_{\bm{x} \in \mathcal{X}} ~ \mathcal{L}\left( \bm{x}; \bm{\lambda}\right)
\end{align}

From the strong convexity of $F\left(\bm{x}\right)$, given $\forall \bm{\lambda} \ge 0$, the Lagrangian $\mathcal{L}\left( \bm{x}; \bm{\lambda}\right)$ is $c_F$-strongly convex in $\bm{x}$ and hence is minimized over $\mathcal{X}$ at a unique point. Let $\bm{x}\left( \bm{\lambda} \right)$ and $D\left( \bm{\lambda} \right)$ denote the optimal solution and value of the inner minimization problem of \eqref{dual_problem}, i.e.,
\begin{align}
\bm{x} \left( \bm{\lambda} \right) & = \arg \min_{\bm{x} \in \mathcal{X}} ~ \mathcal{L}\left( \bm{x}; \bm{\lambda}\right) \label{dual_problem_value} \\
D\left( \bm{\lambda} \right) & = \min_{\bm{x} \in \mathcal{X}} ~ \mathcal{L}\left( \bm{x}; \bm{\lambda}\right) \label{dual_problem_func}
\end{align}

Let $\bm{\lambda}^*$ and $D^*$ denote the optimal solution and value of \eqref{dual_problem}, respectively. Then we have
\begin{align*}
D^* = D\left( \bm{\lambda}^* \right) = \mathcal{L}\left( \bm{x}\left( \bm{\lambda}^* \right); \bm{\lambda}^*\right)
\end{align*}

By Assumption A1 and \cite[Sec. 5.2.3]{boyd2004convex}, the Slater condition of \eqref{primal_problem} holds and the duality gap is zero, i.e., $F^* = D^*$. By the duality theory, the optimal solutions to the primal and dual problems satisfy $\bm{x}^* = \bm{x} \left( \bm{\lambda}^* \right)$. Therefore, to obtain the optimal value and solution of \eqref{primal_problem}, solving its dual problem \eqref{dual_problem} is an alternative method.

The basic idea of dual decomposition is to solve the dual problem in a distributed manner. Note that \eqref{dual_problem_value} and \eqref{dual_problem_func} can be solved separately by agents. For $\forall i \in \mathcal{N}$, define
\begin{align*}
\mathcal{L}_i \left( \bm{x}_i; \bm{\lambda} \right) &= f_i\left(\bm{x}_i\right) + \left< A_i^T \bm{\lambda}, \bm{x}_i \right> \\
\bm{x}_i \left( \bm{\lambda} \right) & = \arg \min_{\bm{x}_i \in \mathcal{X}_i} ~ \mathcal{L}_i \left( \bm{x}_i; \bm{\lambda} \right) \\
D_i \left( \bm{\lambda} \right) & = \min_{\bm{x}_i \in \mathcal{X}_i} ~ \mathcal{L}_i \left( \bm{x}_i; \bm{\lambda}\right) = \mathcal{L}_i \left( \bm{x}_i \left( \bm{\lambda} \right); \bm{\lambda}\right) 
\end{align*}

Hence we have
\begin{align*}
\mathcal{L}\left( \bm{x}; \bm{\lambda}\right) &= \sum_{i\in\mathcal{N}} \mathcal{L}_i \left( \bm{x}_i; \bm{\lambda} \right) - \left<\bm{\lambda}, \bm{b} \right> \\
D\left( \bm{\lambda} \right) &= \sum_{i\in \mathcal{N}} D_i \left( \bm{\lambda} \right) - \left<\bm{\lambda}, \bm{b} \right> \\
\bm{x} \left( \bm{\lambda} \right) &= col\left\{ \bm{x}_i \left( \bm{\lambda} \right) \right\}_{i\in\mathcal{N}}
\end{align*}

The entire distributed algorithm is shown in Algorithm \ref{SyncAlgo}. According to \cite{Steven}, the dual variable ultimately converges to $\bm{\lambda}^*$, and meanwhile, agents obtain the optimal operation decision $\bm{x}^* = col\left\{ \bm{x}_i \left( \bm{\lambda}^* \right) \right\}_{i\in\mathcal{N}}$ by solving subproblems \eqref{subproblem}.

\begin{algorithm}[t]
	\caption{Synchronous and Exact DD-DO Algorithm} 
	\label{SyncAlgo}
	
	\hangafter 1
	\hangindent 1em
	\textbf{Input:} Accuracy tolerance $\epsilon>0$, step size $\alpha>0$, initial dual variable $\bm{\lambda}^0\ge0$, and iteration index $k=0$.
	
	\hangafter 1
	\hangindent 1em
	\textbf{Output:} Optimal operation strategy $\bm{x}^{*}$.
	
	\hangafter 1
	\hangindent 1em
	\textbf{S1 (Solving subproblems):} Agent $i$ obtains its operation strategy $\bm{x}_i^k$ by solving the following subproblem
	\begin{align}
	\min_{\bm{x}_i \in \mathcal{X}_i} ~ f_i\left(\bm{x}_i\right) + \left< A_i^T \bm{\lambda}^k, \bm{x}_i \right> \label{subproblem}
	\end{align}
	
	\hangafter 1
	\hangindent 1em
	\textbf{S2 (Updating dual variable):} The central coordinator updates the dual variable as
	\begin{align} \label{dual_update}
	\bm{\lambda}^{k+1} = \left[ \bm{\lambda}^k + \alpha \left( A \bm{x}^k - \bm{b} \right) \right]^+
	\end{align}
	where $\alpha$ is the step size.
	
	\hangafter 1
	\hangindent 1em
	\textbf{S3:} Evaluate the iterative error $E_k$ as
	\begin{align*} 
	E_k = \left\| \bm{\lambda}^{k+1} - \bm{\lambda}^{k} \right\|
	\end{align*}
	If $E_k \le \epsilon$, $\bm{x}^k$ is recognized as the optimal operation strategy and the algorithm terminates. Otherwise, set $k=k+1$ and go to \textbf{S1}.
	
\end{algorithm}

The direction of updating dual variable \eqref{dual_update} follows the gradient of $D \left( \bm{\lambda} \right)$. Invoking \cite[Prop. 6.1.1]{bertsekas1997nonlinear}, as $\mathcal{X}$ is nonempty and compact, $F\left(\cdot\right)$ is continuous over $\mathcal{X}$, and $\bm{x} \left( \bm{\lambda} \right)$ is the unique optimal solution to \eqref{dual_problem_value}, $D\left( \bm{\lambda} \right)$ is differentiable with the gradient defined as
\begin{align} \label{gradient}
\nabla D\left( \bm{\lambda} \right) = A \bm{x} \left( \bm{\lambda} \right) - \bm{b}
\end{align}
Similarly, $D_i \left( \bm{\lambda} \right)$ is differentiable with the gradient defined as
\begin{align} \label{gradient_i}
\nabla D_i \left( \bm{\lambda} \right) = A_i \bm{x}_i \left( \bm{\lambda} \right)
\end{align}

\begin{remark}
\textit{(Distributed Implementation)} The dual decomposition algorithm can be partially distributed or fully distributed up to the structure of the MAS. Dual decomposition is commonly implemented in a partially distributed manner as Algorithm \ref{SyncAlgo}, where the dual variable is computed by a central coordinator. The algorithm can also be fully distributed depending on the particular sparse communication network, for instance, Internet networks \cite{Steven} and radial distribution grids \cite{magnusson2020distributed}. Without a central coordinator, each agent communicates with its neighbors and updates the dual variable locally. Both the partially and fully distributed algorithms share the same iterative procedure as Algorithm \ref{SyncAlgo}. Hence we follow the partially distributed framework throughout the rest of this paper.

\end{remark}

\subsection{Basic Properties}
To begin with, we analyze the basic properties of dual decomposition. We firstly give a lemma to characterize $\bm{x}_i \left( \bm{\lambda} \right)$ and then give a corollary with respect to $D_i \left( \bm{\lambda} \right)$.

\begin{lemma}
\label{thm:lem1} Suppose Assumption A1 holds. $\bm{x}_i \left( \bm{\lambda} \right)$ is $\left\| A_i \right\|_F / c_i$ - Lipschitz continuous in $\bm{\lambda} \in \mathbb{R}^m_+$.
\end{lemma}

\begin{proof}
For $\forall \bm{\lambda} \in \mathbb{R}^m$, $\bm{x}_i \left( \bm{\lambda} \right)$ is the optimal solution to \eqref{subproblem}. From the KKT condition of the constrained optimization problem \cite[Thm. 3.24]{ruszczynski2006nonlinear}, we have
\begin{align} \label{KKT}
\left< \nabla f_i \left( \bm{x}_i \left( \bm{\lambda} \right) \right) + A_i^T \bm{\lambda}, \bm{y} - \bm{x}_i \left( \bm{\lambda} \right) \right> \ge 0, ~~ \forall \bm{y} \in \mathcal{X}_i
\end{align}

Replacing $\bm{y}$ by $\bm{x}_i \left( \bm{\mu} \right)$, we have
\begin{align} \label{KKT1}
\left< \nabla f_i \left( \bm{x}_i \left( \bm{\lambda} \right) \right) + A_i^T \bm{\lambda}, \bm{x}_i \left( \bm{\mu} \right) - \bm{x}_i \left( \bm{\lambda} \right) \right> \ge 0
\end{align}
and similarly for $\bm{\mu}$
\begin{align} \label{KKT2}
\left< \nabla f_i \left( \bm{x}_i \left( \bm{\mu} \right) \right) + A_i^T \bm{\mu}, \bm{x}_i \left( \bm{\lambda} \right) - \bm{x}_i \left( \bm{\mu} \right) \right> \ge 0
\end{align}

Flipping the signs of the two terms of $\left< \cdot, \cdot \right>$ in \eqref{KKT1} and adding \eqref{KKT2} gives
\begin{align*}
&\underbrace{ \left< \nabla f_i \left( \bm{x}_i \left( \bm{\lambda} \right) \right) - \nabla f_i \left( \bm{x}_i \left( \bm{\mu} \right) \right), \bm{x}_i \left( \bm{\lambda} \right) - \bm{x}_i \left( \bm{\mu} \right) \right> }_{(\Delta_1)}\\
&\quad \quad \quad \quad \quad \quad \quad \quad \quad \le \underbrace{ \left< A_i^T \bm{\mu} - A_i^T \bm{\lambda}, \bm{x}_i \left( \bm{\lambda} \right) - \bm{x}_i \left( \bm{\mu} \right) \right> }_{(\Delta_2)}
\end{align*}

From the strong convexity of $f_i \left( \cdot \right)$, we have
\begin{align*}
c_i \left\| \bm{x}_i \left( \bm{\lambda} \right) - \bm{x}_i \left( \bm{\mu} \right) \right\|^2 \le \Delta_1
\end{align*}

From the Cauchy$-$Schwarz inequality and $\left\| A \right\| \le \left\| A \right\|_F$, we have
\begin{align*}
\Delta_2 \le \left\| A_i \right\|_F \left\| \bm{\lambda} - \bm{\mu} \right\| \left\| \bm{x}_i \left( \bm{\lambda} \right) - \bm{x}_i \left( \bm{\mu} \right) \right\|
\end{align*}

No matter if $\left\| \bm{x}_i \left( \bm{\lambda} \right) - \bm{x}_i \left( \bm{\mu} \right) \right\| = 0$ or not, it immediately follows that
\begin{align*}
\left\| \bm{x}_i \left( \bm{\lambda} \right) - \bm{x}_i \left( \bm{\mu} \right) \right\| \le \frac{\left\| A_i \right\|_F}{c_i} \left\| \bm{\lambda} - \bm{\mu} \right\|
\end{align*}
which completes the proof. 
\end{proof}

From Lemma \ref{thm:lem1}, we can directly give the following corollary with respect to $D \left( \bm{\lambda} \right)$ by the definition \eqref{gradient} and \eqref{gradient_i}.

\begin{corollary} \label{thm:col2}
Suppose Assumption A1 holds. $D(\bm{\lambda} )$ has the following properties. 
\begin{enumerate}
\item $\nabla D_i \left( \bm{\lambda} \right)$ is  $L_i$-Lipschitz continuous with
\begin{align*}
L_i = \left\| A_i \right\|_F^2 / c_i
\end{align*}

\item $\nabla D \left( \bm{\lambda} \right)$ is  $L_D$-Lipschitz continuous with
\begin{align*}
L_D = \left\| A \right\|_F^2 / c_F \ge \sum_{i\in\mathcal{N}} L_i
\end{align*}

\item For $\forall \bm{y} \in \mathbb{R}^n$, we have
\begin{align} \label{Corollary1_3}
- \left< \bm{y}, \nabla^2 D \left( \bm{\lambda} \right) \bm{y} \right> \le L_D \left\| \bm{y} \right\|^2
\end{align}
\end{enumerate}
\end{corollary}

\section{ Asynchronous and Inexact DD-DO Algorithm}
In this section, we formulate the asynchrony in communication and the inexactness in solving subproblems, and propose the asynchronous and inexact DD-DO algorithm.

\subsection{Asynchrony}
Asynchrony commonly exists in realistic  MASs during iteration of the distributed algorithm \cite{wang2020asynchronous2}. Due to time delay, packet drop, and different computational rates, individual agents have to wait for the slowest information, which lengthens the idle time. Severe asynchrony in communication may lead to slow convergence speed or even divergence. The asynchronous dual decomposition algorithm is firstly studied in \cite{Steven}. Each agent (the central coordinator) solves the subproblem (updates the dual variable) with the previously stored information, if the latest is not received.

We follow the formulation of asynchrony as \cite{Steven}. The local clock $\mathcal{K}_i$ ($\mathcal{K}_D$) is the set of time slots when agent $i$ (the central coordinator) takes action, while the global clock $\mathcal{K} = \mathcal{K}_D \cup \mathcal{K}_1 \cup ... \cup \mathcal{K}_{\left|\mathcal{N}\right|}$ is the union of all local clocks.

Noting that the idle time during iteration should be limited, we make the following assumption on asynchrony.

\textit{Assumption A2:} There exists an asynchrony parameter $k_0 \ge 0$ such that for  $\forall k \in \mathcal{K}$, the central coordinator and agents receive information at least once during the interval $\left[k-k_0,k\right]$.

Let $\widehat{\bm{\lambda}}^k$ denote the previously stored dual variable that agent $i$ uses to solve its subproblem at time slot $k \in \mathcal{K}_i$. Similarly, denote by $\widehat{\bm{x}}_i^k$ the previously stored primal variable that the central coordinator uses to update the dual variable.

Under Assumption A2, we have
\begin{align*}
\widehat{\bm{\lambda}}^k & = \bm{\lambda}^{k - \delta_{di} \left(k\right)}, ~~~ 0 \le \delta_{di} \left(k\right) \le k_0 \\
\widehat{\bm{x}}_i^k &= \widetilde{\bm{x}}_i^{k-\delta_{pi} \left(k\right)} , ~~~ 0 \le \delta_{pi} \left(k\right) \le k_0
\end{align*}

It should be noted that, when inexactness is considered, the situation will turn out to be more complicated, as we will discuss in Section III-C.

\subsection{Inexactness}
Inexactness in solving individual subproblems is another crucial issue that may deteriorate the performance of DD-DO algorithms. Slight errors of solving subproblems could accumulate during iteration, which may cause severe inexactness of the algorithm.

We follow the formulation of inexactness given in \cite{necoara2013rate}. Let $\widetilde{\bm{x}}_i \left( \bm{\lambda} \right)$ denote the inexact solution to  subproblem $i$, and let
\begin{align} \label{D_tilde}
\widetilde{D}_i \left( \bm{\lambda} \right) = \mathcal{L}_i \left( \widetilde{\bm{x}}_i \left( \bm{\lambda} \right); \bm{\lambda} \right) = f_i\left(\widetilde{\bm{x}}_i \left( \bm{\lambda} \right)\right) + \left< \bm{\lambda}, A_i\widetilde{\bm{x}}_i \left( \bm{\lambda} \right) \right>
\end{align}
denote the inexact value of the subproblem.

Before analysis, we make the following assumption.

\textit{Assumption A3:} Given $\forall \bm{\lambda} \in \mathbb{R}^m_+$, an inexact solution $\widetilde{\bm{x}}_i \left( \bm{\lambda} \right) \in \mathcal{X}_i$ is obtained. There exists a local inexactness parameter $\varepsilon_i \ge 0$ such that $\forall \bm{\lambda} \in \mathbb{R}^m_+$
\begin{align*}
\left| D_i \left( \bm{\lambda} \right) - \widetilde{D}_i \left( \bm{\lambda} \right) \right| \le \varepsilon_i
\end{align*}
i.e., $\left| \mathcal{L}_i \left( \bm{x}_i \left( \bm{\lambda} \right); \bm{\lambda} \right) - \mathcal{L}_i \left( \widetilde{\bm{x}}_i \left( \bm{\lambda} \right); \bm{\lambda} \right) \right| \le \varepsilon_i$.

Under Assumption A3, we have the following lemma. 
\begin{lemma}
\label{thm:lem3} Suppose Assumptions A1 and A3 hold. The distance between the optimal and inexact solutions is bounded by
\begin{align*}
\left\| \bm{x}_i \left( \bm{\lambda} \right) - \widetilde{\bm{x}}_i \left( \bm{\lambda} \right) \right\|^2 \le \frac{2 \varepsilon_i}{c_i}
\end{align*}
\end{lemma}

\begin{proof}
Given $\bm{\lambda} \in \mathcal{R}^m_+$, from the strong convexity of $\mathcal{L}_i \left( \cdot; \bm{\lambda} \right)$ in $\bm{x}_i$, we have
\begin{align*}
&\mathcal{L}_i \left( \widetilde{\bm{x}}_i \left( \bm{\lambda} \right); \bm{\lambda} \right) \ge \mathcal{L}_i \left( \bm{x}_i \left( \bm{\lambda} \right); \bm{\lambda} \right) + \frac{c_i}{2} \left\| \widetilde{\bm{x}}_i \left( \bm{\lambda} \right) - \bm{x}_i \left( \bm{\lambda} \right) \right\|^2 \\
& \quad \quad \quad \quad \quad \quad + \underbrace{ \left< \nabla f_i \left( \bm{x}_i \left( \bm{\lambda} \right) \right) + A_i^T \bm{\lambda}, \widetilde{\bm{x}}_i \left( \bm{\lambda} \right) - \bm{x}_i \left( \bm{\lambda} \right) \right> }_{(\Delta_3)}
\end{align*}

From \eqref{KKT}, $(\Delta_3)$ is non-negative, which completes the proof. 
\end{proof}

The inexactness of subproblems can be extended to the whole problem directly. Denote the inexactness parameter by $\varepsilon_D = \sum_{i\in \mathcal{N}} \varepsilon_i$. Recalling $c_F = \min_{i\in\mathcal{N}} c_i$, we have
\begin{align*}
& \left| D \left( \bm{\lambda} \right) - \widetilde{D} \left( \bm{\lambda} \right) \right| \le \sum_{i\in \mathcal{N}} \left| D_i \left( \bm{\lambda} \right) - \widetilde{D}_i \left( \bm{\lambda} \right) \right| \le \varepsilon_D \\
& \left\| \bm{x} \left( \bm{\lambda} \right) - \widetilde{\bm{x}} \left( \bm{\lambda} \right) \right\|^2 = \sum_{i\in \mathcal{N}} \left\| \bm{x}_i \left( \bm{\lambda} \right) - \widetilde{\bm{x}}_i \left( \bm{\lambda} \right) \right\|^2 \le \frac{2 \varepsilon_D}{c_F}
\end{align*}



\subsection{Asynchronous and Inexact DD-DO Algorithm}

\begin{algorithm}[t]
	\caption{Asynchronous and Inexact DD-DO Algorithm} 
	\label{AsyncAlgo}
	
	\hangafter 1
	\hangindent 1em
	\textbf{Input:} Accuracy tolerance $\epsilon>0$, step size $\alpha>0$, initial dual variable $\bm{\lambda}^0\ge0$, and iteration index $k=0$.
	
	\hangafter 1
	\hangindent 1em
	\textbf{Output:} Suboptimal operation strategy $\widetilde{\bm{x}}^{*}$.
	
	\hangafter 1
	\hangindent 1em
	\textbf{S1 (Solving subproblems):} If $k \in \mathcal{K}_i$, agent $i$ solves its subproblem and obtains a suboptimal operation strategy $\widetilde{\bm{x}}_i^k$, which satisfies
	\begin{subequations}
	\begin{align}
	& \left\| \widetilde{\bm{x}}_i^k - \bm{x}_i \left( \widehat{\bm{\lambda}}^{k,i} \right) \right\|^2 \le \frac{2 \varepsilon_i}{c_i} \\
	& \bm{x}_i \left( \widehat{\bm{\lambda}}^{k,i} \right) = \arg \min_{\bm{x}_i \in \mathcal{X}_i} f_i\left(\bm{x}_i\right) + \left< A_i^T \widehat{\bm{\lambda}}^{k,i}, \bm{x}_i \right> \label{inexact_sub} \\
	& \widehat{\bm{\lambda}}^{k,i} = \bm{\lambda}^{k - \delta_{di} \left(k\right)}, ~~~ 0 \le \delta_{di} \left(k\right) \le k_0 \label{asyn_primal_var}
	\end{align}
	\end{subequations}
	where $\delta_{di} \left(k\right)$ is the time delay with respect to $k$. Otherwise, $\widetilde{\bm{x}}_i^k = \widetilde{\bm{x}}_i^{k-1}$ holds.

	\hangafter 1
	\hangindent 1em
	\textbf{S2 (Updating dual variable):} If $k \in \mathcal{K}_D$, the central coordinator updates the dual variable as
	\begin{subequations}
	\begin{align}
	\bm{\lambda}^{k+1} & = \left[ \bm{\lambda}^k + \alpha \bm{\nu}^k \right]^+ \\
	\bm{\nu}^k &= A \widehat{\bm{x}}^k - \bm{b} \label{nu_def} \\
	\widehat{\bm{x}}_i^k &= \widetilde{\bm{x}}_i^{k-\delta_{pi} \left(k\right)} , ~~~ 0 \le \delta_{pi} \left(k\right) \le k_0 \label{asyn_dual_var} 
	\end{align}
	\end{subequations}
	where $\bm{\nu}^k$ is the estimated gradient and $\delta_{pi} \left(k\right)$ is the time delay with respect to $k$. Otherwise, $\bm{\lambda}^{k+1} = \bm{\lambda}^k$ holds.
	
	\hangafter 1
	\hangindent 1em
	\textbf{S3:} Evaluate the iterative error $E_k$ as
	\begin{align*} 
	E_k = \left\| \bm{\lambda}^{k+1} - \bm{\lambda}^{k} \right\|
	\end{align*}
	If $k \in \mathcal{K}_D$ and $E_k \le \epsilon$, the algorithm is regarded to converge and the iteration terminates. Otherwise, set $k=k+1$ and go to \textbf{S1}.
	
\end{algorithm}

Algorithm \ref{AsyncAlgo} gives an asynchronous and inexact version of Algorithm \ref{SyncAlgo}. 
Here, the information flow of $\bm{x}$ and $\bm{\lambda}$ can be further rewritten as
\begin{align*}
& \widehat{\bm{x}}_i^k = \widetilde{\bm{x}}_i^{k - \delta_{pi}\left(k\right)} \\
& \left\| \widetilde{\bm{x}}_i^{k - \delta_{pi}\left(k\right)} - \bm{x}_i \left( \widehat{ \bm{\lambda}} ^{k - \delta_{pi}\left(k\right),i} \right) \right\|^2 \le \frac{2 \varepsilon_i}{c_i} \\
& \bm{x}_i \left( \widehat{ \bm{\lambda}} ^{k - \delta_{pi}\left(k\right),i} \right) = \bm{x}_i \left( \bm{\lambda} ^{k - \delta_{pi}\left(k\right) - \delta_{di}\left(k - \delta_{pi}\left(k\right)\right)} \right)
\end{align*}

For simplicity, define $\delta_i^k := \delta_{pi}\left(k\right) + \delta_{di}\left(k - \delta_{pi}\left(k\right)\right)$ and $\widetilde{\bm{x}}_i ( \bm{\lambda} ^{k - \delta_i^k} ) := \widehat{\bm{x}}_i^k$. Under Assumption A2, $0 \le \delta_i^k \le 2k_0$ holds. Then we have
\begin{subequations} 
\begin{align}
& \widehat{\bm{x}}_i^k = \widetilde{\bm{x}}_i \left( \bm{\lambda} ^{k - \delta_i^k} \right), ~ 0 \le \delta_i^k \le 2k_0 \label{asyn_inext_defin} \\
& \left\| \widetilde{\bm{x}}_i \left( \bm{\lambda} ^{k - \delta_i^k} \right) - \bm{x}_i \left( \bm{\lambda} ^{k - \delta_i^k} \right) \right\|^2 \le \frac{2 \varepsilon_i}{c_i}
\end{align}
\end{subequations}

\section{Main Result}
In this section, we analyze the convergence of the asynchronous and inexact DD-DO algorithm.

\subsection{Bound of the Sum-of-Square of Dual Deviations}
Define the dual deviation $\bm{\sigma}^k := \bm{\lambda}^{k+1} - \bm{\lambda}^k$ and the sum-of-square of dual deviations $S^k := \sum_{\kappa = 0}^k \left\| \bm{\sigma}^\kappa \right\|^2$. We turn to prove that $\sqrt{S^k}$ increases not faster than $\mathcal{O} (\sqrt{k})$, starting with the following two lemmas.

\begin{lemma}
\label{thm:lem4} In Algorithm 2,  $\forall k\in \mathcal{K}$
\begin{align} \label{Lemma3}
\left< \bm{\nu}^k, \bm{\sigma}^k \right> \ge \left\| \bm{\sigma}^k \right\|^2/\alpha, ~~ \forall k \in \mathcal{K}
\end{align}
\end{lemma}
\begin{proof}
If $k\in\mathcal{K}_D$, by $\bm{\lambda}^{k+1} = \left[ \bm{\lambda}^k + \alpha \bm{\nu}^k \right]^+$ and the projection theorem \cite[Prop. 2.1.3]{bertsekas1997nonlinear}, we have
\begin{align*}
\left< \bm{\lambda}^{k+1} - \bm{\lambda}^k - \alpha \bm{\nu}^k, \bm{\lambda} - \bm{\lambda}^{k+1} \right> \ge 0, ~~ \forall \bm{\lambda} \ge 0
\end{align*}

By replacing $\bm{\lambda} = \bm{\lambda}^k$ and recalling $\bm{\sigma}^k = \bm{\lambda}^{k+1} - \bm{\lambda}^k$, we obtain \eqref{Lemma3} directly. 

If $k\notin\mathcal{K}_D$, this lemma holds trivially since $\bm{\lambda}^{k+1} = \bm{\lambda}^k$. This completes the proof.
\end{proof}

\begin{lemma}
\label{thm:lem5} Suppose Assumptions A1-A3 hold. In Algorithm 2, we have for $\forall k\in \mathcal{K}$
\begin{align} \label{Lemma4}
\left\| \nabla D\left(\bm{\lambda}^k\right) - \bm{\nu}^k \right\| \le L_D \sum_{\kappa=k-2k_0}^{k-1} \left\| \bm{\sigma}^\kappa \right\| + \sqrt{2L_D\varepsilon_D}
\end{align}
\end{lemma}

\begin{proof}
From the definitions of $\nabla D\left(\bm{\lambda}^k\right)$ and $\bm{\nu}^k$, we have
\begin{align*}
&\left\| \nabla D\left(\bm{\lambda}^k\right) - \bm{\nu}^k \right\| \\
\le& \left\|A\right\| \left\| \bm{x}\left(\bm{\lambda}^k\right) - \widehat{\bm{x}}^k \right\| \\
\le& \left\|A\right\|_F \left\| col\left\{ \bm{x}_i\left(\bm{\lambda}^k\right) - \bm{x}_i\left(\bm{\lambda}^{k-\delta_i^k}\right) \right\}_{i\in\mathcal{N}} \right\| \\
& + \left\|A\right\|_F \left\| col\left\{ \bm{x}_i\left(\bm{\lambda}^{k-\delta_i^k}\right) - \widetilde{\bm{x}}_i\left(\bm{\lambda}^{k-\delta_i^k}\right) \right\}_{i\in\mathcal{N}} \right\| \\
=& \left\|A\right\|_F \sqrt{\sum_{i\in\mathcal{N}} \left\| \bm{x}_i\left(\bm{\lambda}^k\right) - \bm{x}_i\left(\bm{\lambda}^{k-\delta_i^k}\right) \right\|^2} \\
& + \left\|A\right\|_F \sqrt{\sum_{i\in\mathcal{N}} \left\|  \bm{x}_i\left(\bm{\lambda}^{k-\delta_i^k}\right) - \widetilde{\bm{x}}_i\left(\bm{\lambda}^{k-\delta_i^k}\right) \right\|^2} \\
\le& \left\|A\right\|_F \sqrt{\sum_{i\in\mathcal{N}} \frac{\left\| A_i \right\|_F^2}{c_i^2} \left\| \bm{\lambda}^k - \bm{\lambda}^{k-\delta_i^k} \right\|^2} + \left\|A\right\|_F \sqrt{\frac{2 \varepsilon_D}{c_F}} \\
\le& \left\|A\right\|_F \sqrt{\sum_{i\in\mathcal{N}} \frac{\left\| A_i \right\|_F^2}{c_F^2} \sum_{\kappa=k-\delta_i^k}^{k-1} \left\| \bm{\sigma}^\kappa \right\|^2} + \sqrt{2L_D\varepsilon_D} \\
\le& \left\|A\right\|_F \sqrt{ \frac{\left\| A \right\|_F^2}{c_F^2} \sum_{\kappa=k-2k_0}^{k-1} \left\| \bm{\sigma}^\kappa \right\|^2} + \sqrt{2L_D\varepsilon_D} \\
\le& L_D \sum_{\kappa=k-2k_0}^{k-1} \left\| \bm{\sigma}^\kappa \right\| + \sqrt{2L_D\varepsilon_D}
\end{align*}
where the second inequality follows from the definition \eqref{asyn_inext_defin}, the triangle inequality and $\left\|A\right\| \le \left\|A\right\|_F$; the third one yields from Lemmas \ref{thm:lem1} and \ref{thm:lem3}; recalling $c_F = \min_{i\in\mathcal{N}} c_i$, $\bm{\sigma}^k = \bm{\lambda}^{k+1} - \bm{\lambda}^k$, and $L_D = \left\| A \right\|_F^2 / c_F$, the fourth holds from the triangle inequality; the fifth one follows from Assumption A2 and $\left\|A\right\|_F = \sum_{i\in\mathcal{N}} \left\|A_i\right\|_F$; the last yields from the definition of $L_D$ and $\sqrt{x+y} \le \sqrt{x} + \sqrt{y}, ~ \forall x,y\ge0$.
\end{proof}

Then we obtain the upper bound of $\sqrt{S^k}$, if the step size is sufficiently small. For simplification, denote $D^0 := D\left(\bm{\lambda}^0\right)$.

\begin{theorem}
\label{thm:thm1} Suppose Assumptions A1-A3 hold. In Algorithm 2, provided that the step size satisfies
\begin{align} \label{Step-Size_1}
0 < \alpha < \frac{1}{\left(2k_0+1/2\right)L_D}
\end{align}
then there is
\begin{align} \label{Theorem_1}
\sqrt{S^k} \le \frac{\sqrt{2L_D\varepsilon_D}}{\gamma_\alpha} \sqrt{k+1} + 2\sqrt{\frac{D^* - D^0}{\gamma_\alpha}}
\end{align}
where $\gamma_\alpha$ is a positive constant parameter defined as
\begin{align*}
\gamma_\alpha: = \frac{1}{\alpha} - \left(2k_0+\frac{1}{2}\right)L_D
\end{align*}
\end{theorem}
\begin{proof}
Note that $\bm{\lambda}^{k+1}=\bm{\lambda}^k+\bm{\sigma}^k $. Considering the second-order Taylor expansion, there exists $\bm{\varphi} \in \mathcal{R}^m_+$ such that 
\begin{align*}
& D\left(\bm{\lambda}^k\right) - D\left(\bm{\lambda}^{k+1}\right) \\
= & -\left<\nabla D\left(\bm{\lambda}^k\right), \bm{\sigma}^k \right> - \frac{1}{2} \left< \bm{\sigma}^k, \nabla^2 D\left(\bm{\varphi}\right) \bm{\sigma}^k \right> \\
\stackrel{\eqref{Corollary1_3}}{\le} & \left< \bm{\nu}^k - \nabla D\left(\bm{\lambda}^k\right), \bm{\sigma}^k \right> - \left< \bm{\nu}^k, \bm{\sigma}^k \right> + \frac{L_D}{2} \left\|\bm{\sigma}^k\right\|^2 \\
\stackrel{\eqref{Lemma3}}{\le} & \left\| \bm{\nu}^k - \nabla D\left(\bm{\lambda}^k\right) \right\| \left\| \bm{\sigma}^k \right\| + \left( \frac{L_D}{2}-\frac{1}{\alpha} \right) \left\| \bm{\sigma}^k \right\|^2 \\
\stackrel{\eqref{Lemma4}}{\le} & L_D \sum_{\kappa=k-2k_0}^{k-1} \big\| \bm{\sigma}^\kappa \big\| \left\| \bm{\sigma}^k \right\| + \sqrt{2L_D\varepsilon_D} \left\| \bm{\sigma}^k \right\| \\
& + \left( \frac{L_D}{2}-\frac{1}{\alpha} \right) \left\| \bm{\sigma}^k \right\|^2 \\
\le & \frac{L_D}{2} \sum_{\kappa=k-2k_0}^{k-1} \left\{ \big\| \bm{\sigma}^\kappa \big\|^2 + \left\| \bm{\sigma}^k \right\|^2 \right\}  + \sqrt{2L_D\varepsilon_D} \left\| \bm{\sigma}^k \right\| \\
& + \left( \frac{L_D}{2}-\frac{1}{\alpha} \right) \left\| \bm{\sigma}^k \right\|^2 \\
= & \frac{L_D}{2} \sum_{\kappa=k-2k_0}^{k} \big\| \bm{\sigma}^\kappa \big\|^2 + \sqrt{2L_D\varepsilon_D} \left\| \bm{\sigma}^k \right\| \\
& + \left( k_0L_D - \frac{1}{\alpha} \right) \left\| \bm{\sigma}^k \right\|^2 
\end{align*}
where the last inequality holds from the mean value inequality.

Summing over $k$, we have
\begin{align*}
& D^0 - D\left(\bm{\lambda}^{k+1}\right) \\
\le & \frac{L_D}{2} \sum_{\kappa=0}^{k} \sum_{k'=\kappa-2k_0}^{\kappa} \left\| \bm{\sigma}^{k'} \right\|^2 + \sqrt{2L_D\varepsilon_D} \sum_{\kappa=0}^{k} \left\| \bm{\sigma}^\kappa \right\| \\
& + \left( k_0L_D - \frac{1}{\alpha} \right) S^k \\
\le & \left( \left(2k_0+\frac{1}{2}\right)L_D - \frac{1}{\alpha} \right) S^k + \sqrt{2\left(k+1\right) L_D\varepsilon_D} \sqrt{S^k}
\end{align*}
where the second inequality yields from $\frac{a+b}{2} \le \sqrt{\frac{a^2+b^2}{2}}$.

Then we have
\begin{align*}
\gamma_\alpha S^k - \sqrt{2\left(k+1\right) L_D\varepsilon_D} \sqrt{S^k} \le D\left(\bm{\lambda}^{k+1}\right) - D^0
\end{align*}

Assume the step-size $\alpha$ satisfies \eqref{Step-Size_1}. Noting the optimality of $D^*$, we obtain the result by solving the above quadratic inequality and using $\sqrt{x+y} \le \sqrt{x} + \sqrt{y}, ~ \forall x,y\ge0$. 
\end{proof}

\begin{remark}
\textit{(Interaction between Asynchrony and Inexactness)} $\sqrt{S^k}$ characterizes the interaction between asynchrony and inexactness in the DD-DO algorithm. If the solutions to subproblems are exact, i.e., $\varepsilon_D = 0$, the $\mathcal{O} (\sqrt{k})$ term in \eqref{Theorem_1} vanishes. In other words, $\sqrt{S^k}$ is not greater than a positive constant and hence $\lim_{k\to\infty} \left\| \bm{\sigma}^k \right\| = 0$, which implies the convergence of $\left\{ \bm{\lambda}^k \right\}$ in the asynchronous DD-DO algorithm,  as proved in \cite{Steven}. However, if the solutions to subproblems are inexact, i.e., $\varepsilon_D > 0$, the errors will be  accumulated, leading to the increasing of $\sqrt{S^k}$ in $\mathcal{O} (\sqrt{k})$, and hence $\left\{ \bm{\lambda}^k \right\}$ fails to converge. Simulation results in Section V verify this theorem.
\end{remark}

\subsection{Bound of the Norm of Dual Variable}
Here we show that in Algorithm \ref{AsyncAlgo}, the norm of dual variable $\left\| \bm{\lambda}^k \right\|$ increases not faster than $\mathcal{O} (\sqrt{k})$. We start the analysis from the following lemma.

\begin{lemma}
\label{thm:lem7} Suppose Assumptions A1 and A3 hold. In Algorithm 2, we have for $\forall \bm{\lambda},\bm{\mu} \in \mathbb{R}^m_+$
\begin{align*}
0 \le \widetilde{D}_i\left(\bm{\lambda}\right) - D_i \left(\bm{\mu}\right) + \left< \bm{\mu} - \bm{\lambda}, A_i \widetilde{\bm{x}}_i \left( \bm{\lambda} \right) \right> \le L_i \left\| \bm{\mu} - \bm{\lambda} \right\|^2 + 2\varepsilon_i
\end{align*}
\end{lemma}

\begin{proof}
For the left-hand side inequality, from the optimality of $\bm{x}_i \left(\cdot\right)$, we have
\begin{align*}
D_i \left(\bm{\mu}\right) =& \min_{\bm{x}_i \in \mathcal{X}_i} ~ f_i\left(\bm{x}_i\right) + \left< \bm{\mu}, A_i\bm{x}_i \right> \\
=& f_i \left( \bm{x}_i \left( \bm{\mu} \right) \right) + \left< \bm{\mu}, A_i \bm{x}_i \left( \bm{\mu} \right) \right> \\
\le& f_i \left( \widetilde{\bm{x}}_i \left( \bm{\lambda} \right) \right) + \left< \bm{\mu}, A_i \widetilde{\bm{x}}_i \left( \bm{\lambda} \right) \right> \\
=& \widetilde{D}_i \left(\bm{\lambda}\right) + \left< \bm{\mu} - \bm{\lambda}, A_i \widetilde{\bm{x}}_i \left( \bm{\lambda} \right) \right>
\end{align*}

For the right-hand side inequality, by the Lipschitz continuity of $\nabla D_i \left(\cdot\right)$, we have
\begin{align*}
D_i \left(\bm{\mu}\right) \ge& D_i \left(\bm{\lambda}\right) + \left< \bm{\mu} - \bm{\lambda}, A_i \bm{x}_i \left( \bm{\lambda} \right) \right> - \frac{L_i}{2} \left\| \bm{\mu} - \bm{\lambda} \right\|^2 \\
=& \widetilde{D}_i \left(\bm{\lambda}\right) + \left< \bm{\mu} - \bm{\lambda}, A_i \widetilde{\bm{x}}_i \left( \bm{\lambda} \right) \right> - \frac{L_i}{2} \left\| \bm{\mu} - \bm{\lambda} \right\|^2 \\
&+ \left< \bm{\mu} - \bm{\lambda}, A_i \left( \bm{x}_i \left( \bm{\lambda} \right) - \widetilde{\bm{x}}_i \left( \bm{\lambda} \right) \right) \right> + D_i\left(\bm{\lambda}\right) - \widetilde{D}_i\left(\bm{\lambda}\right)\\
\ge& \widetilde{D}_i \left(\bm{\lambda}\right) + \left< \bm{\mu} - \bm{\lambda}, A_i \widetilde{\bm{x}}_i \left( \bm{\lambda} \right) \right> - \frac{L_i}{2} \left\| \bm{\mu} - \bm{\lambda} \right\|^2 \\
&- \left\| A_i \right\|_F \left\| \bm{\mu} - \bm{\lambda} \right\| \left\| \bm{x}_i \left( \bm{\lambda} \right) - \widetilde{\bm{x}}_i \left( \bm{\lambda} \right) \right\| - \varepsilon_i\\
\ge& \widetilde{D}_i \left(\bm{\lambda}\right) + \left< \bm{\mu} - \bm{\lambda}, A_i \widetilde{\bm{x}}_i \left( \bm{\lambda} \right) \right> - L_i \left\| \bm{\mu} - \bm{\lambda} \right\|^2 \\
&- \frac{c_i}{2} \left\| \bm{x}_i \left( \bm{\lambda} \right) - \widetilde{\bm{x}}_i \left( \bm{\lambda} \right) \right\|^2  - \varepsilon_i \\
\ge& \widetilde{D}_i \left(\bm{\lambda}\right) + \left< \bm{\mu} - \bm{\lambda}, A_i \widetilde{\bm{x}}_i \left( \bm{\lambda} \right) \right> - L_i \left\| \bm{\mu} - \bm{\lambda} \right\|^2 - 2\varepsilon_i
\end{align*}
where the second inequality yields from the Cauchy$-$Schwarz inequality, $\left\| A_i \right\| \le \left\| A_i \right\|_F$ and Assumption A3; the third one follows from the mean value inequality as
\begin{align*}
& \left\| A_i \right\|_F \left\| \bm{\mu} - \bm{\lambda} \right\| \left\| \bm{x}_i \left( \bm{\lambda} \right) - \widetilde{\bm{x}}_i \left( \bm{\lambda} \right) \right\| \\
= & \frac{ \left\| A_i \right\|_F }{ \sqrt{c_i} } \left\| \bm{\mu} - \bm{\lambda} \right\| \cdot \sqrt{c_i} \left\| \bm{x}_i \left( \bm{\lambda} \right) - \widetilde{\bm{x}}_i \left( \bm{\lambda} \right) \right\| \\
\le & \frac{ \left\| A_i \right\|_F^2 }{ 2 c_i } \left\| \bm{\mu} - \bm{\lambda} \right\|^2 + \frac{c_i}{2} \left\| \bm{x}_i \left( \bm{\lambda} \right) - \widetilde{\bm{x}}_i \left( \bm{\lambda} \right) \right\|^2 \\
= & \frac{L_i}{2} \left\| \bm{\mu} - \bm{\lambda} \right\|^2 + \frac{c_i}{2} \left\| \bm{x}_i \left( \bm{\lambda} \right) - \widetilde{\bm{x}}_i \left( \bm{\lambda} \right) \right\|^2
\end{align*}
and the last inequality holds from Lemma \ref{thm:lem3}.
\end{proof}

Then we have the following theorem of bounded dual variables.


\begin{theorem} \label{thm:thm2}
Suppose Assumptions A1-A3 hold. In Algorithm 2, if the step size satisfies
\begin{align} \label{Step_Size_2}
0 \le \alpha \le \frac{1}{2L_D}
\end{align}
then the dual variable $\bm{\lambda}^{k+1}$ is bounded by
\begin{align*}
\left\| \bm{\lambda}^{k+1} \right\| \le 2\left\| \bm{\lambda}^* \right\| + \left\| \bm{\lambda}^0 \right\| + 2\sqrt{\alpha \varepsilon_D} \sqrt{k+1} + \sqrt{2k_0} \sqrt{S^k}
\end{align*}
\end{theorem}

\begin{proof}
For $k\in \mathcal{K}_D$, applying the projection theorem \cite[Prop. 2.1.3]{bertsekas1997nonlinear}, we have
\begin{align} \label{Proj_Ineq}
\left< \bm{\lambda}^{k+1} - \bm{\lambda}^k - \alpha \bm{\nu}^k, \bm{\lambda} - \bm{\lambda}^{k+1} \right> \ge 0, ~~ \forall \bm{\lambda} \ge 0
\end{align}
and hence
\begin{align}
& \left\| \bm{\lambda}^{k+1} - \bm{\lambda} \right\|^2 \notag \\
= & \left\| \bm{\lambda}^{k+1} - \bm{\lambda}^k + \bm{\lambda}^k - \bm{\lambda} \right\|^2 \notag \\
= & \left\| \bm{\lambda}^k - \bm{\lambda} \right\|^2 + \left\| \bm{\lambda}^{k+1} - \bm{\lambda}^k \right\|^2 + 2\left< \bm{\lambda}^{k+1} - \bm{\lambda}^k, \bm{\lambda}^k - \bm{\lambda} \right> \notag \\
= & \left\| \bm{\lambda}^k - \bm{\lambda} \right\|^2 - \left\| \bm{\lambda}^{k+1} - \bm{\lambda}^k \right\|^2 + 2\left< \bm{\lambda}^{k+1} - \bm{\lambda}^k, \bm{\lambda}^{k+1} - \bm{\lambda} \right> \notag \\
\le & \left\| \bm{\lambda}^k - \bm{\lambda} \right\|^2 - 2\alpha L_D \left\| \bm{\lambda}^{k+1} - \bm{\lambda}^k \right\|^2 + 2\alpha \left< \bm{\lambda}^{k+1} - \bm{\lambda}, \bm{\nu}^k \right> \notag \\
\le & \left\| \bm{\lambda}^k - \bm{\lambda} \right\|^2 - 2\alpha \left< \bm{\lambda}^{k+1} - \bm{\lambda}, \bm{b}\right> \notag \\
+ & 2\alpha \sum_{i\in \mathcal{N}} \left\{ \left< \bm{\lambda}^{k+1} - \bm{\lambda}, A_i \widehat{\bm{x}}_i^k \right> - L_i \left\| \bm{\lambda}^{k+1} - \bm{\lambda}^k \right\|^2 \right\} \label{Theorem_Ineq_1}
\end{align}
where the first inequality follows from \eqref{Step_Size_2} and \eqref{Proj_Ineq}; the last one holds from the decomposition over $i$ and the definition of $\bm{\nu}^k$ in \eqref{nu_def}.

For $\forall i\in \mathcal{N}$ we have
\begin{align}
&\left< \bm{\lambda}^{k+1} - \bm{\lambda}, A_i \widehat{\bm{x}}_i^k \right> - L_i \left\| \bm{\lambda}^{k+1} - \bm{\lambda}^k \right\|^2 \notag \\
\le& \underbrace{ \left< \bm{\lambda}^{k+1} - \bm{\lambda}^{k-\delta_i^k}, A_i \widetilde{\bm{x}}_i \left(\bm{\lambda}^{k-\delta_i^k}\right) \right> - L_i \left\| \bm{\lambda}^{k+1} - \bm{\lambda}^{k-\delta_i^k} \right\|^2 }_{(\Delta_4)} \notag \\
& - \underbrace{ \left< \bm{\lambda} - \bm{\lambda}^{k-\delta_i^k}, A_i \widetilde{\bm{x}}_i \left(\bm{\lambda}^{k-\delta_i^k}\right) \right> }_{(\Delta_5)} + L_i \left\| \bm{\lambda}^k - \bm{\lambda}^{k-\delta_i^k} \right\|^2 \notag \\
\le& D_i \left( \bm{\lambda}^{k+1} \right) - \widetilde{D}_i \left( \bm{\lambda}^{k-\delta_i^k} \right) + 2\varepsilon_i \notag \\
& + \widetilde{D}_i \left( \bm{\lambda}^{k-\delta_i^k} \right) - D_i \left( \bm{\lambda} \right) + L_i \sum_{\kappa=k-\delta_i^k}^{k-1} \left\| \bm{\sigma}^\kappa \right\|^2 \notag \\
\le& D_i \left( \bm{\lambda}^{k+1} \right) - D_i \left( \bm{\lambda} \right) + 2\varepsilon_i + L_i \sum_{\kappa=k-2k_0}^{k-1} \left\| \bm{\sigma}^\kappa \right\|^2 \label{Theorem_Ineq_2}
\end{align}
where the first inequality holds from the triangle inequality; Lemma \ref{thm:lem7} is utilized in $(\Delta_4)$ and $(\Delta_5)$ in the second inequality; the last one follows from Assumption A2.

Following \eqref{Theorem_Ineq_1}, we have
\begin{align}
&\left\| \bm{\lambda}^{k+1} - \bm{\lambda} \right\|^2 - \left\| \bm{\lambda}^k - \bm{\lambda} \right\|^2 \notag \\
\stackrel{\eqref{Theorem_Ineq_2}}{\le}& - 2\alpha \left< \bm{\lambda}^{k+1} - \bm{\lambda}, \bm{b}\right> + 2\alpha \sum_{i\in \mathcal{N}} \Bigg\{ D_i \left( \bm{\lambda}^{k+1} \right) - D_i \left( \bm{\lambda} \right) \notag \\
& + 2\varepsilon_i + L_i \sum_{\kappa=k-2k_0}^{k-1} \left\| \bm{\sigma}^\kappa \right\|^2 \Bigg\} \notag \\
\stackrel{\eqref{Step_Size_2}}{\le}& 2\alpha \left( D \left( \bm{\lambda}^{k+1} \right) - D \left( \bm{\lambda} \right) \right) + 4\alpha \varepsilon_D + \sum_{\kappa=k-2k_0}^{k-1} \left\| \bm{\sigma}^\kappa \right\|^2 \label{Theorem_Ineq_3}
\end{align}

By replacing $\bm{\lambda} = \bm{\lambda}^*$ and noting that $D \left( \bm{\lambda}^* \right) \ge D \left( \bm{\lambda}^{k+1} \right)$, we have
\begin{align*}
\left\| \bm{\lambda}^{k+1} - \bm{\lambda}^* \right\|^2 \le \left\| \bm{\lambda}^k - \bm{\lambda}^* \right\|^2 + 4\alpha \varepsilon_D + \sum_{\kappa=k-2k_0}^{k-1} \left\| \bm{\sigma}^\kappa \right\|^2
\end{align*}

Note that the above inequality also holds when $k\notin \mathcal{K}_D$. Summing over $k$, we obtain
\begin{align*}
\left\| \bm{\lambda}^{k+1} - \bm{\lambda}^* \right\|^2 \le \left\| \bm{\lambda}^0 - \bm{\lambda}^* \right\|^2 + 4\alpha \left(k+1\right) \varepsilon_D + 2k_0S^k
\end{align*}

By decomposing the quadratic terms, using $\left< \bm{\lambda}^*, \bm{\lambda}^{k+1} \right> \le \left\| \bm{\lambda}^* \right\| \left\| \bm{\lambda}^{k+1} \right\|$, and noting $\left< \bm{\lambda}^0, \bm{\lambda}^* \right> \ge 0$, we get
\begin{align*}
\left\| \bm{\lambda}^{k+1} \right\|^2 - 2\left\| \bm{\lambda}^* \right\| \left\| \bm{\lambda}^{k+1} \right\| \le \left\| \bm{\lambda}^0 \right\|^2 + 4\alpha \left(k+1\right) \varepsilon_D + 2k_0S^k
\end{align*}

Solving the quadratic inequality and using $\sqrt{x+y} \le \sqrt{x} + \sqrt{y}, ~ \forall x,y\ge0$, we complete the proof. 
\end{proof}

\subsection{Convergence Analysis}
Based on Theorems \ref{thm:thm1} and \ref{thm:thm2}, we show that the asynchronous and inexact DD-DO algorithm converges in $\mathcal{O} (1 / \sqrt{k})$.

Instead of the primal sequence $\left\{ \widehat{\bm{x}}^k \right\}$ and the dual sequences $\left\{ \bm{\lambda}^k \right\}$, we consider their running averages over iteration named as the primal and dual average variables, which are defined as
\begin{align*}
\bar{\bm{x}}^k := \frac{1}{\left| \mathcal{K}_D^k \right|} \sum_{\kappa \in \mathcal{K}_D^k} \widehat{\bm{x}}^{\kappa}, ~ ~
\bar{\bm{\lambda}}^{k+1} := \frac{1}{\left| \mathcal{K}_D^k \right|} \sum_{\kappa \in \mathcal{K}_D^k} \bm{\lambda}^{\kappa+1}
\end{align*}
where $\mathcal{K}_D^k := \left\{ \kappa \in \mathcal{K}_D ~|~ \kappa \le k \right\}$.

Average is widely utilized in iterative algorithms \cite{necoara2013rate, devolder2014first, nedic2009distributed, mateos2016distributed}. Intuitively, the violation of constraints and the deviation of objectives can be reduced by averaging due to the convexity of constraints and the primal objective and the concavity of the dual objective by \cite{nedic2009approximate}.

We have the following theorem of convergence.
\begin{theorem} \label{thm:thm3}
Suppose Assumptions A1-A3 hold. If the step size satisfies
\begin{align}
0 < \alpha < \min \left\{ \frac{1}{\left(2k_0+1/2\right)L_D}, \frac{1}{2L_D} \right\}\label{Step_Size}
\end{align}
Algorithm 2 has the following convergence performance:
\begin{enumerate}[a)]
\item The violation of constraints is bounded by
\begin{align} \label{Result_Cons}
\left\| \left[ A \bar{\bm{x}}^k - \bm{b} \right]^+ \right\| \le \frac{M_{1/2}}{\sqrt{k+1}} + \frac{M_1}{k+1} 
\end{align}

\item The deviation of primal value is bounded by
\begin{align} \label{Result_Primal}
- \frac{M_{1/2}\left\| \bm{\lambda}^* \right\|}{\sqrt{k+1}} - \frac{M_1 \left\| \bm{\lambda}^* \right\|}{k+1} &\le F\left( \bar{\bm{x}}^k \right) - F^* \notag \\
& \le  N_0 + \frac{N_{1/2}}{\sqrt{k+1}} + \frac{N_1}{k+1}
\end{align}

\item The deviation of dual value is bounded by
\begin{align} \label{Result_Dual}
0 \le D^* - D\left( \bar{\bm{\lambda}}^{k+1} \right) &\le N_0 + \frac{N_{1/2}}{\sqrt{k+1}} + \frac{N_1'}{k+1}
\end{align}

\item The deviation of primal average variable is bounded by
\begin{align} \label{Result_Solution}
\left\| \bar{\bm{x}}^k - \bm{x}^* \right\|^2 \le \frac{2N_0}{c_F} + \frac{2\left(N_{1/2} + \left\|\bm{\lambda}^*\right\| M_{1/2}\right)}{c_F \sqrt{k+1}} \notag \\
+ \frac{2\left(N_1 + \left\|\bm{\lambda}^*\right\| M_1\right)}{c_F \left(k+1\right)}
\end{align}
\end{enumerate}
where $M_{1/2}, M_1, N_0, N_{1/2}, N_1$, and $N_1'$ are positive constants defined as
\begin{align*}
M_{1/2} &:= 2\sqrt{\alpha \varepsilon_D} + \frac{2 \sqrt{k_0 L_D \varepsilon_D}}{\gamma_\alpha} \\
M_1 &:= \frac{k_0+1}{\alpha} \left( 2\left\| \bm{\lambda}^* \right\| + \left\| \bm{\lambda}^0 \right\| + 2\sqrt{\frac{2k_0 \left( D^* - D^0 \right)}{\gamma_\alpha}} \right) \\
N_0 &:= 2\varepsilon_D + \frac{2k_0\left(k_0+1\right) L_D \varepsilon_D}{\alpha \gamma_\alpha^2} \\
N_{1/2} &:= \frac{4k_0(k_0+1) \sqrt{2 L_D \varepsilon_D \left( D^* - D^0 \right)}}{\alpha \gamma_\alpha^{3/2}} \\
N_1 &:= \frac{k_0+1}{2\alpha} \left( \left\| \bm{\lambda}^0 \right\|^2 + \frac{8k_0\left( D^* - D^0 \right)}{\gamma_\alpha} \right) \\
N_1' &:= \frac{k_0+1}{2\alpha} \left( \left\| \bm{\lambda}^0 - \bm{\lambda}^* \right\|^2 + \frac{8k_0\left( D^* - D^0 \right)}{\gamma_\alpha} \right)
\end{align*}
\end{theorem}

\begin{proof}
a) If $k\in\mathcal{K}_D$, we have
\begin{align*}
\bm{\lambda}^{k+1} = \left[ \bm{\lambda}^k + \alpha \left( A \widehat{\bm{x}}^k - \bm{b} \right) \right]^+ \ge \bm{\lambda}^k + \alpha \left( A \widehat{\bm{x}}^k - \bm{b} \right)
\end{align*}

Summing over $k\in\mathcal{K}_D$, we obtain
\begin{align*}
\alpha \sum_{\kappa \in \mathcal{K}_D^k} \left( A \widehat{\bm{x}}^\kappa - \bm{b} \right) \le \bm{\lambda}^{k+1} - \bm{\lambda}^0 \le \bm{\lambda}^{k+1}
\end{align*}

By taking average and norm, we have
\begin{align*}
&\left\| \left[ A \bar{\bm{x}}^k - \bm{b} \right]^+ \right\| \le \frac{\left\| \bm{\lambda}^{k+1} \right\|}{\alpha \left| \mathcal{K}_D^k \right|} \le \frac{k_0+1}{\alpha \left(k+1\right)} \left\| \bm{\lambda}^{k+1} \right\|
\end{align*}

Invoking Theorems \ref{thm:thm1} and \ref{thm:thm2} immediately yields \eqref{Result_Cons}.

b) For the left-hand side inequality of \eqref{Result_Primal},  the optimality of $F^*$ and $\bm{\lambda}^* \ge 0$ yields
\begin{align*}
F^* \le & F\left( \bar{\bm{x}}^k \right) + \left< \bm{\lambda}^*, A \bar{\bm{x}}^k - \bm{b} \right> \\
\le & F\left( \bar{\bm{x}}^k \right) + \Big\| \bm{\lambda}^* \Big\| \left\| \left[ A \bar{\bm{x}}^k - \bm{b} \right]^+ \right\| 
\end{align*}

Applying \eqref{Result_Cons} establishes left-hand side inequality.

For the right-hand side inequality of \eqref{Result_Primal}, by taking $\bm{\lambda} = 0$ in \eqref{Theorem_Ineq_1} we have at $k\in\mathcal{K}_D$
\begin{align*}
&\left\| \bm{\lambda}^{k+1} \right\|^2 - \left\| \bm{\lambda}^k \right\|^2 + 2\alpha \left< \bm{\lambda}^{k+1}, \bm{b} \right> \\
\le & 2\alpha \sum_{i\in \mathcal{N}} \left\{ \left< \bm{\lambda}^{k+1}, A_i \widehat{\bm{x}}_i^k \right> - L_i \left\| \bm{\lambda}^{k+1} - \bm{\lambda}^k \right\|^2 \right\} \\
\le & 2\alpha \sum_{i\in \mathcal{N}} \Big\{ \underbrace{\left< \bm{\lambda}^{k-\delta_i^k}, A_i \widetilde{\bm{x}}_i \left(\bm{\lambda}^{k-\delta_i^k}\right) \right>}_{(\Delta_6)} + L_i \left\| \bm{\lambda}^k - \bm{\lambda}^{k-\delta_i^k} \right\|^2\\
+ & \underbrace{\left< \bm{\lambda}^{k+1} - \bm{\lambda}^{k-\delta_i^k}, A_i \widetilde{\bm{x}}_i \left(\bm{\lambda}^{k-\delta_i^k}\right) \right> - L_i \left\| \bm{\lambda}^{k+1} - \bm{\lambda}^{k-\delta_i^k} \right\|^2}_{(\Delta_7)} \Big\} \\
\le & 2\alpha \sum_{i\in \mathcal{N}} \Big\{ \widetilde{D}_i \left( \bm{\lambda}^{k-\delta_i^k} \right) - f_i \left( \widetilde{\bm{x}}_i \left(\bm{\lambda}^{k-\delta_i^k}\right) \right) \\
&+ L_i \sum_{\kappa=k-2k_0}^{k-1} \left\| \bm{\sigma}^\kappa \right\|^2 + D_i \left( \bm{\lambda}^{k+1} \right) - \widetilde{D}_i \left( \bm{\lambda}^{k-\delta_i^k} \right) + 2\varepsilon_i \Big\} \\
\le & 2\alpha \sum_{i\in \mathcal{N}} D_i \left( \bm{\lambda}^{k+1} \right) - 2\alpha F\left( \widehat{\bm{x}}^k \right) + 4\alpha \varepsilon_D + \sum_{\kappa=k-2k_0}^{k-1} \left\| \bm{\sigma}^\kappa \right\|^2
\end{align*}
where the second inequality follows from the triangle inequality; the definition of $\widetilde{D}_i \left( \cdot \right)$ in \eqref{D_tilde} is utilized in $(\Delta_6)$; the right-hand side inequality of Lemma \ref{thm:lem7} is used in $(\Delta_7)$; the last inequality holds from \eqref{Step_Size_2} and $F ( \widehat{\bm{x}}^k ) = \sum_{i\in \mathcal{N}} f_i ( \widetilde{\bm{x}}_i ( \bm{\lambda}^{k-\delta_i^k} ) )$.

Note that 
\begin{align*}
\sum_{i\in \mathcal{N}} D_i \left( \bm{\lambda}^{k+1} \right) - \left< \bm{\lambda}^{k+1}, \bm{b} \right> = D \left( \bm{\lambda}^{k+1} \right) \le D^* = F^*
\end{align*}

Hence we have
\begin{align*}
\frac{\left\| \bm{\lambda}^{k+1} \right\|^2 - \left\| \bm{\lambda}^k \right\|^2}{2\alpha} \le  F^* - F ( \widehat{\bm{x}}^k ) + 2\varepsilon_D + \sum_{\kappa=k-2k_0}^{k-1} \frac{\left\| \bm{\sigma}^\kappa \right\|^2}{2\alpha}
\end{align*}

Summing over $k\in\mathcal{K}_D$, taking average, and using the convexity of $F$ and Theorem \ref{thm:thm1}, we have the right-hand side inequality of \eqref{Result_Primal} as
\begin{align*}
F\left( \bar{\bm{x}}^k \right) \le & F^* + 2\varepsilon_D + \frac{k_0+1}{2\alpha (k+1)} \left( \left\| \bm{\lambda}^0 \right\|^2 + 2k_0 S^k \right) \\ 
\le & F^* + N_0 + \frac{N_{1/2}}{\sqrt{k+1}} + \frac{N_1}{k+1}
\end{align*}

c) From the optimality of $D^*$, the left-hand size of \eqref{Result_Dual} holds directly. For the right-hand side, by replacing $\bm{\lambda} = \bm{\lambda}^* $ in \eqref{Theorem_Ineq_3}, we have at $k\in\mathcal{K}_D$
\begin{align*}
&\left\| \bm{\lambda}^{k+1} - \bm{\lambda}^* \right\|^2 - \left\| \bm{\lambda}^k - \bm{\lambda}^* \right\|^2 \\
&\le 2\alpha \left( D\left(\bm{\lambda}^{k+1}\right) - D^* \right) + 4\alpha\varepsilon_D + \sum_{\kappa=k-2k_0}^{k-1} \left\| \bm{\sigma}^\kappa \right\|^2
\end{align*}

Summing over $k\in\mathcal{K}_D$, taking average, and using the concavity of $D$ and Theorem \ref{thm:thm1}, we have
\begin{align*}
D^* - D\left( \bar{\bm{\lambda}}^{k+1} \right) \le & 2\varepsilon_D + \frac{k_0+1}{2\alpha (k+1)} \left( \left\| \bm{\lambda}^0 - \bm{\lambda}^* \right\|^2 + 2k_0 S^k \right) \\ 
\le & N_0 + \frac{N_{1/2}}{\sqrt{k+1}} + \frac{N_1'}{k+1}
\end{align*}

d) By \cite[Thm. 3.25]{ruszczynski2006nonlinear}, we have the KKT condition of problem \eqref{primal_problem} as
\begin{subequations}
\begin{align}
& \left< \nabla F\left(\bm{x}^*\right) + A^T\bm{\lambda}^*, \bm{x} - \bm{x}^*\right> \ge 0, ~~ \forall \bm{x} \in \mathcal{X} \label{KKT_a} \\
& A\bm{x}^* \le \bm{b}, ~~ \bm{\lambda}^* \ge 0 \\
& \left< \bm{\lambda}^*, A\bm{x}^* \right> = 0 \label{KKT_c}
\end{align}
\end{subequations}

By replacing $\bm{x} = \bar{\bm{x}}^k$ in \eqref{KKT_a}, we obtain
\begin{align}
& - \left< \nabla F\left(\bm{x}^*\right), \bar{\bm{x}}^k - \bm{x}^*\right> \notag \\
\le & \left< A^T\bm{\lambda}^*, \bar{\bm{x}}^k - \bm{x}^*\right> \notag \\
= & \left< \bm{\lambda}^*, A \bar{\bm{x}}^k - \bm{b} \right> - \left< \bm{\lambda}^*, A  \bm{x}^* - \bm{b} \right> \notag \\
= & \left< \bm{\lambda}^*, A \bar{\bm{x}}^k - \bm{b} \right> \label{nosense_ref}
\end{align}

From the strong convexity of $F\left( \cdot \right)$, we have
\begin{align*}
\frac{c_F}{2} \left\| \bar{\bm{x}}^k - \bm{x}^* \right\|^2 &\le F\left( \bar{\bm{x}}^k \right) - F^* - \left< \nabla F\left(\bm{x}^*\right), \bar{\bm{x}}^k - \bm{x}^*\right> \\
&\stackrel{\eqref{nosense_ref}}{\le} F\left( \bar{\bm{x}}^k \right) - F^* + \left< \bm{\lambda}^*, A \bar{\bm{x}}^k - \bm{b} \right> \\
&\le F\left( \bar{\bm{x}}^k \right) - F^* + \Big\| \bm{\lambda}^* \Big\| \left\| \left[ A \bar{\bm{x}}^k - \bm{b} \right]^+ \right\| 
\end{align*}

The proof is completed by nothing \eqref{Result_Cons} and the right-hand side inequality in \eqref{Result_Primal}. 
\end{proof}

\begin{remark}
\textit{(Generality)} Theorem \ref{thm:thm3} indicates the uniform ultimate boundedness (UUB) of the  DD-DO algorithm under asynchrony and inexactness. 
The convergence results considering only asynchrony \cite{Steven} or inexactness \cite{necoara2013rate} can be regarded as special cases of our result by simply setting the inexactness or asynchrony parameter as zero. On the one hand, if the algorithm is synchronous ($k_0 = 0$) and inexact ($\varepsilon_D > 0$), $N_{1/2}$ vanishes, while other parameters decrease, similarly to \cite{necoara2013rate}. It indicates that asynchrony commonly slows down the convergence and magnifies errors. On the other hand, if the algorithm is asynchronous ($k_0 > 0$) and exact ($\varepsilon_D = 0$), $M_1, N_1$ and $N_1'$ are unchanged, while the rest parameters are zero. Here, the algorithm converges to the optimal solution, similarly to \cite{Steven}. 

Moreover, we show that the asynchronous algorithm converges in $\mathcal{O} ( 1 / k )$, which, to the best of our knowledge, has not been presented in the existing literature \cite{alkano2017asynchronous, bolognani2014distributed, magnusson2020distributed, Steven, lee2015convergence, notarnicola2017distributed}.

\begin{corollary}
Suppose Assumptions A1-A3 hold. If the step size satisfies \eqref{Step_Size}, the asynchronous and exact version of Algorithm 2 has the following convergence performance:
\begin{align*}
\left\| \left[ A \bar{\bm{x}}^k - \bm{b} \right]^+ \right\| &\le  \frac{M_1}{k+1} \\
- \frac{M_1 \left\| \bm{\lambda}^* \right\|}{k+1} \le F\left( \bar{\bm{x}}^k \right) - F^* &\le \frac{N_1}{k+1} \\
0 \le D^* - D\left( \bar{\bm{\lambda}}^{k+1} \right) &\le \frac{N_1'}{k+1} \\
\left\| \bar{\bm{x}}^k - \bm{x}^* \right\|^2 &\le \frac{2\left(N_1 + \left\|\bm{\lambda}^*\right\| M_1\right)}{c_F \left(k+1\right)}
\end{align*}
\end{corollary}
\end{remark}

\section{Illustrative Example}
In this section, the convergence results of the asynchronous and inexact DD-DO algorithm are demonstrated by numerical simulations carried on a 6-agent system.

\subsection{Overview of Implementation}
The simulation is carried on a desktop with Intel i7-10700 CPU and 16 GB memory. The simulation platform is MATLAB 2016B, and commercial solver CPLEX \cite{website:CPLEX} is utilized to solve subproblems with the intermediary toolbox YALMIP {website:YALMIP}.

The dual decomposition algorithm is applied to the network utility maximization (NUM, detailed description can be found in \cite{Chiang1, Chiang2, Chiang3}) problem. There are $n$ sources (agents) connected by $m$ links, where agent $i$ wants to maximize its utility $U_i \left( x_i \right)$ with respect to the resource transmission rate $x_i$ through the given static path. The system congestion is the maximal transmission capacity $b_j$ of every link $j$. The NUM problem is formulated as
\begin{align*}
\min_{\bm{x}} ~ & F\left(\bm{x}\right) = - \sum_{i} U_i\left(x_i\right) = - \sum_{i} C_b - C_a \left( x_i - \overline{x}_i \right)^2 \\
\text{s.t.} ~ & x_i \in \mathcal{X}_i = \left\{ y ~|~ \underline{x}_i \le y \le \overline{x}_i \right\}, ~ \forall i \\
& A \bm{x} \le \bm{b}
\end{align*}
where the matrix $A$ indicates the topology of the network as
\begin{align*}
A_{j,i} = \left\{
\begin{array}{l}
1,  \quad \text{sourse} ~i~ \text{goes through link} ~ j\\
0,  \quad \text{otherwise}
\end{array} \right.
\end{align*}

Consider a MAS with 6 sources and 7 links, whose topology is presented in Fig. \ref{fig_Topology}. The parameters of agents are provided in Table \ref{Agent_Data}. The capacities of links are as follows
\begin{align*}
\bm{b} = \left[ 15, 17, 20, 15, 20, 20, 15 \right]^T
\end{align*}

\begin{figure}[!t]
\centering
\includegraphics[width=0.35\textwidth]{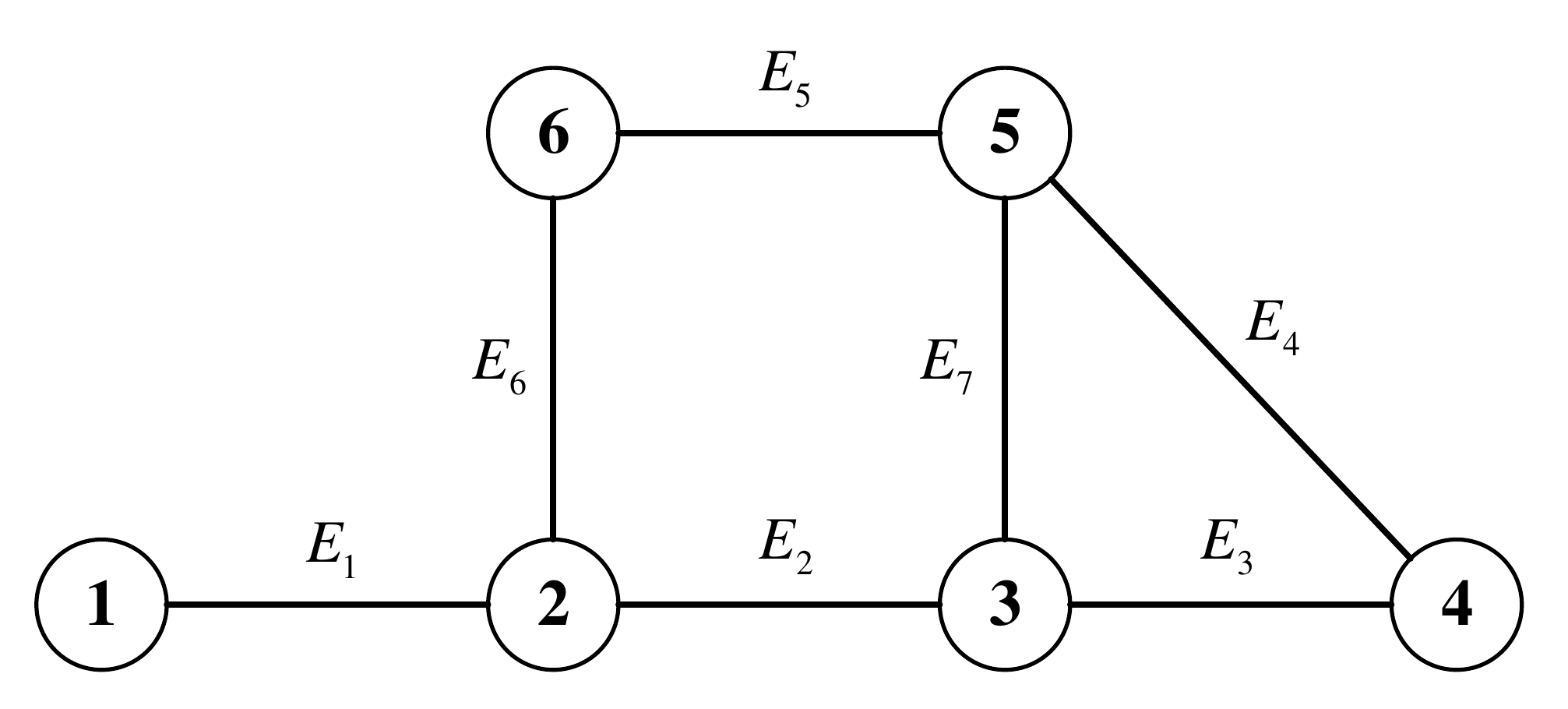}
\caption{The topology of the 6-agent network.}
\label{fig_Topology}
\end{figure}

\begin{table}[!t]
\centering
\caption{Agent Data}
\label{Agent_Data}
\begin{tabular}{ c c c c c c }
\hline
Agent & $\underline{x}_s$ & $\overline{x}_s$ & $C_a$ & $C_b$ & Path \\
\hline
1	&	0	&	5.9	&	1.8	&	62.658	&	$E_1 \to E_2 \to E_7$\\
2	&	0	&	6.6	&	2.2	&	95.832	&	$E_6 \to E_5 \to E_4$\\
3	&	0	&	7.5	&	2.7	&	151.875	&	$E_7 \to E_5 \to E_6 \to E_1$\\
4	&	0	&	4.8	&	3.5	&	80.640	&	$E_3 \to E_2 \to E_6 \to E_5$\\
5	&	0	&	5.4	&	1.2	&	34.992	&	$E_4 \to E_3 \to E_2 \to E_1$\\
6	&	0	&	8.1	&	0.5	&	32.805	&	$E_6 \to E_2 \to E_3 \to E_4$\\
\hline
\end{tabular}
\end{table}

To describing the asynchrony, we generate local clocks by the function \textit{randi} in MATLAB, where the time interval follows the pseudo discrete uniform distribution within $\left[ 0, k_0-1 \right]$. In terms of inexactness, we obtain the exact optimal value $D_i \left(\bm{\lambda}\right)$ by CPLEX, and then search for the inexact solution (randomly pick one if not unique) as
\begin{align*}
\widetilde{\bm{x}}_i \left( \bm{\lambda} \right) \in \arg\max_{\bm{x}_i} ~ & \left| \mathcal{L}_i \left( \bm{x}_i ; \bm{\lambda} \right) - D_i \left(\bm{\lambda}\right) \right| \\
\text{s.t.} ~ & \left| \mathcal{L}_i \left( \bm{x}_i ; \bm{\lambda} \right) - D_i \left(\bm{\lambda}\right) \right| \le \varepsilon_i
\end{align*}
which is the worst case and hence representative.

\subsection{Convergence}
The asynchronous and inexact DD-DO algorithm is implemented in the NUM problem. Fig. \ref{fig_Cons} to \ref{fig_X} are the curves of the violation of constraints, the relative errors of the primal and dual objectives, and the deviation of the primal variable during iteration, respectively. The horizontal coordinate is the iterative index and the vertical coordinate is the corresponding value, both in the exponential form. 

Keeping the step size $\alpha = 0.004$ and changing the asynchrony parameter ($k_0 = 0 / 4$) and the inexactness parameter ($\varepsilon_D = 0 / 30$), we obtain 4 types of curves in blue (synchronous and exact), red (asynchronous and exact), green (synchronous and inexact), and yellow (asynchronous and inexact). The solid lines are the real iterative curves, while the dashed lines are the upper bounds calculated by \eqref{Result_Cons} - \eqref{Result_Solution}.

\begin{figure}[!t]
	\centering
	\includegraphics[width=0.5\textwidth]{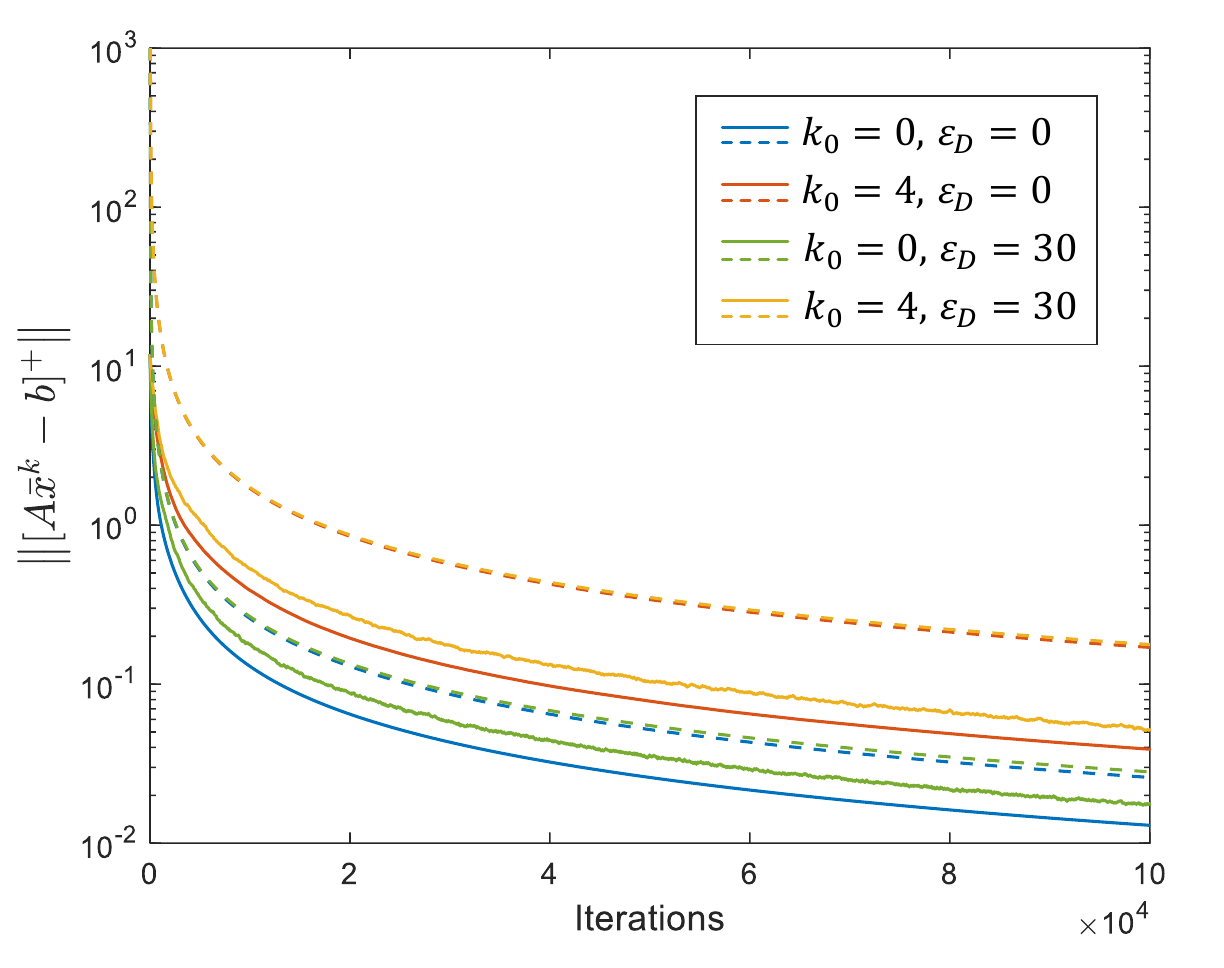}
	\caption{The violation of constraints during iteration.}
	\label{fig_Cons}
\end{figure}

\begin{figure}[!t]
	\centering
	\includegraphics[width=0.5\textwidth]{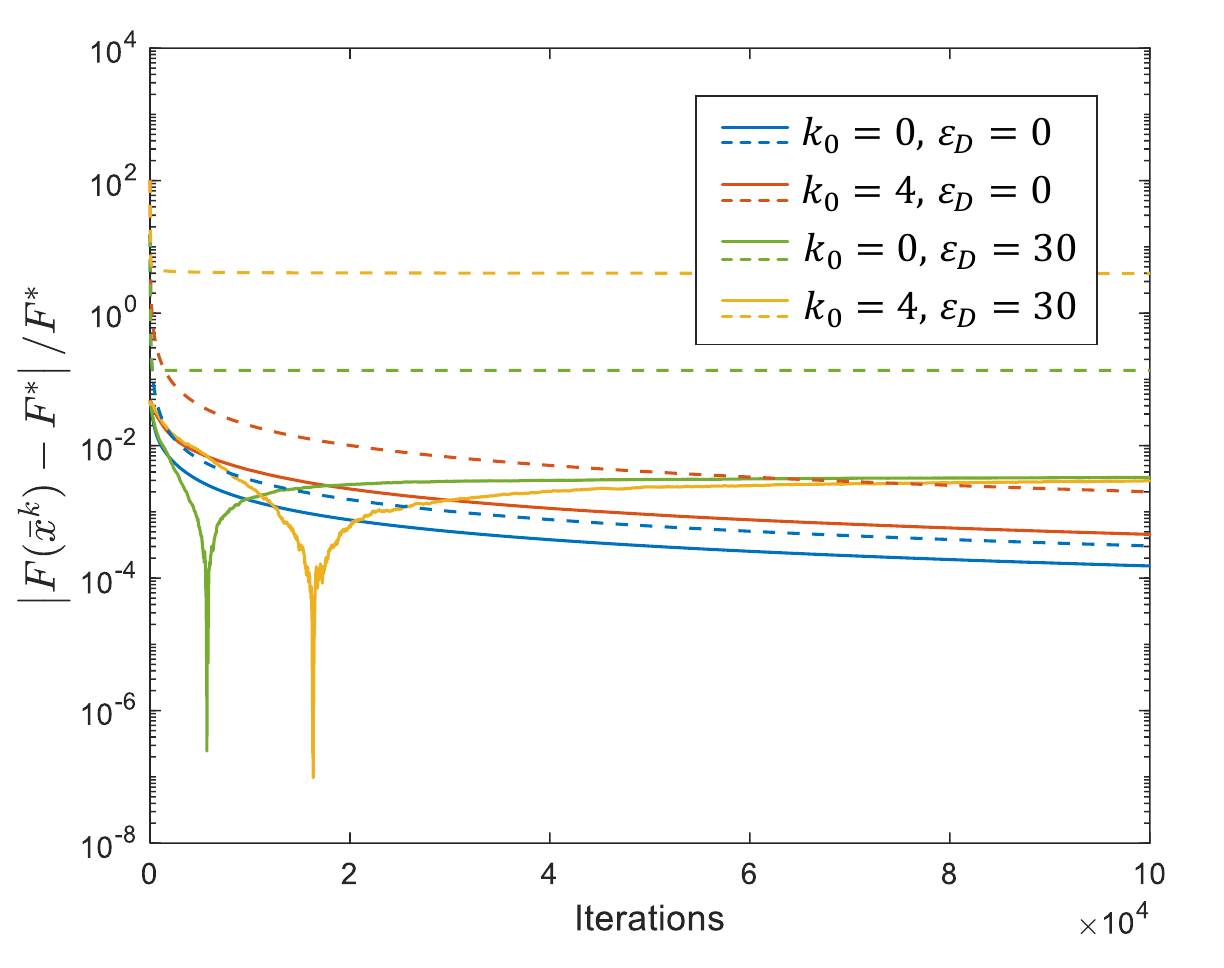}
	\caption{The relative error of the primal objective during iteration.}
	\label{fig_Primal}
\end{figure}

\begin{figure}[!t]
	\centering
	\includegraphics[width=0.5\textwidth]{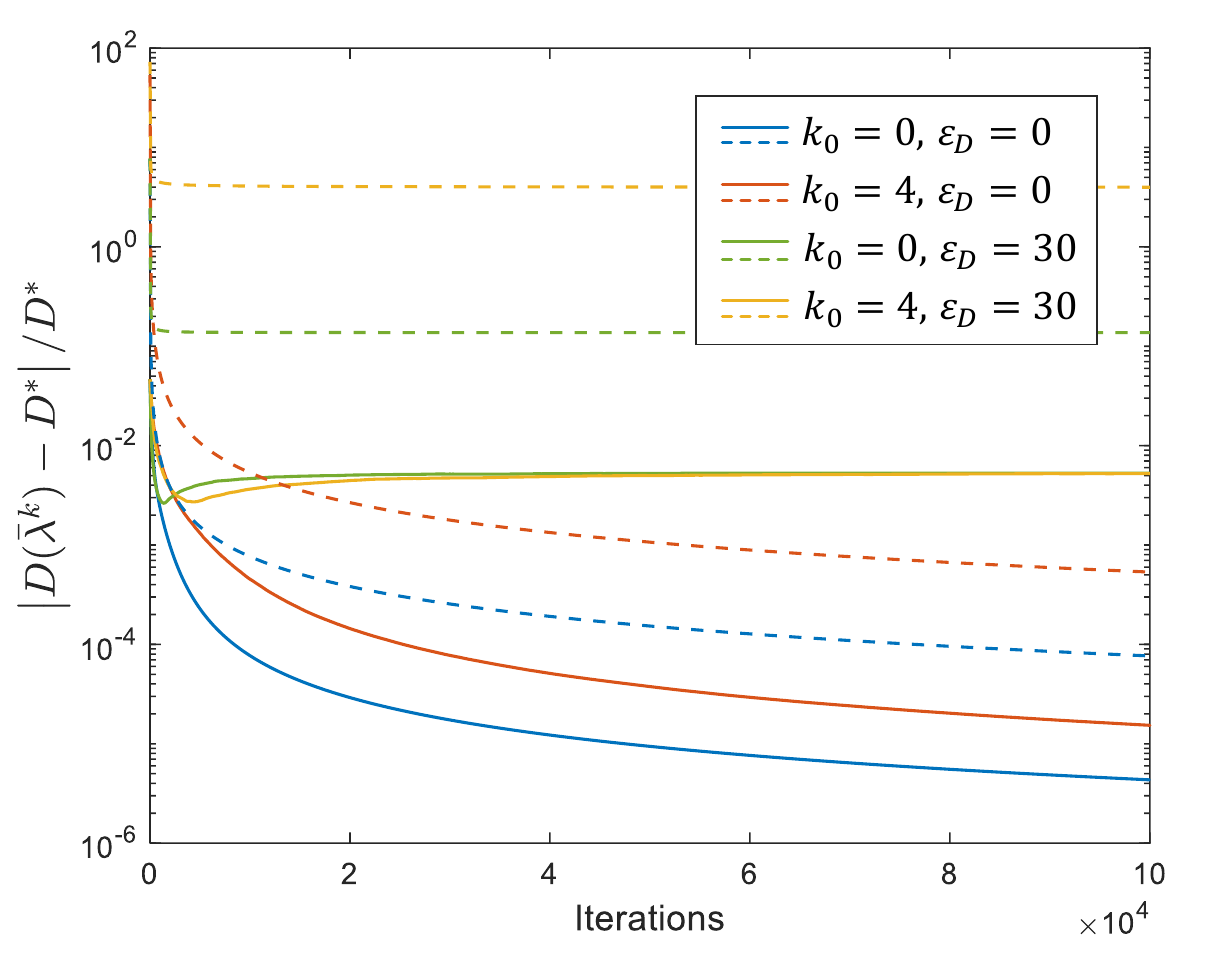}
	\caption{The relative error of the dual objective during iteration.}
	\label{fig_Dual}
\end{figure}

\begin{figure}[!t]
	\centering
	\includegraphics[width=0.5\textwidth]{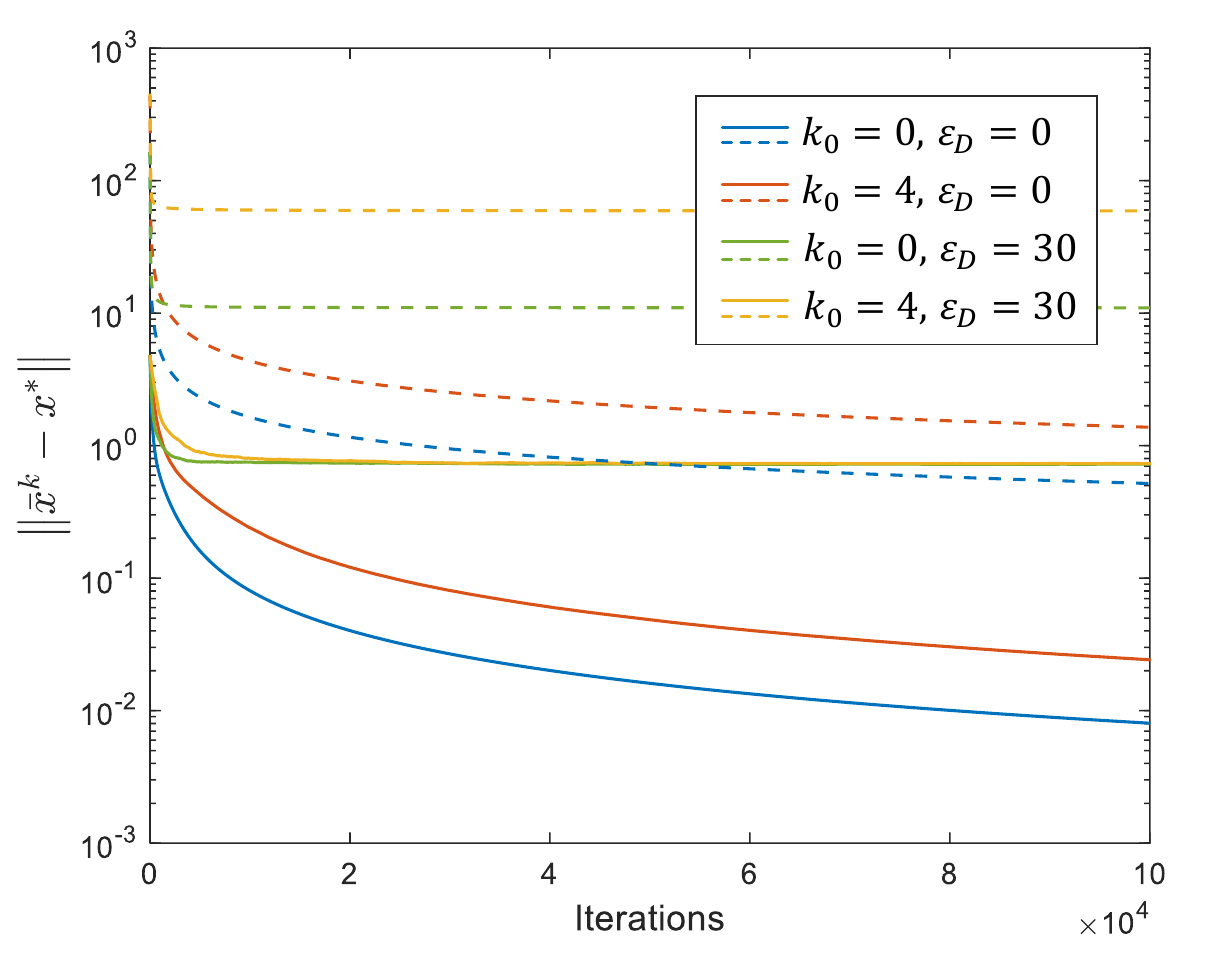}
	\caption{The deviation of the primal variable during iteration.}
	\label{fig_X}
\end{figure}

The violation of constraints during iteration is shown in Fig. \ref{fig_Cons}. The curves all converge, even if the scenario is inexact or/and asynchronous. It verifies the convergence result \eqref{Result_Cons} in Theorem \ref{thm:thm3}, i.e., the feasibility is always satisfied. From this figure, we also find that both asynchrony and inexactness increase the violation of constraints.

Fig. \ref{fig_Primal} and Fig. \ref{fig_Dual} are the relative errors of the primal and dual objectives during iteration. The comparison of real and dashed curves verifies the convergence results \eqref{Result_Primal} and \eqref{Result_Dual} in Theorem \ref{thm:thm3}, i.e., the value of objective converges to a neighborhood of the optimal value. In Fig. \ref{fig_Primal}, the severe fluctuations of the green and yellow curves result from the transformation of $F\left( \bar{\bm{x}}^k \right) - F^*$ from a negative value to a positive one. It is clear that the inexactness has a remarkable influence on the upper bounds on the ranges of neighborhoods, while the convergence speed slows down from synchronous to asynchronous scenarios.

The deviation of the primal variable from the optimal solution is shown in Fig. \ref{fig_X}. The simulation results verify our theoretic analysis on the solution accuracy \eqref{Result_Solution} in Theorem \ref{thm:thm3}. If the solutions to subproblems are inexact, the primal variable sequence converges to a neighborhood of the optimal solution. Otherwise, the deviation decreases to zero and the optimal solution can be obtained. On the other hand, asynchrony slows down the convergence of the primal variable, similarly to the influence on the objective values, as indicated in Theorem \ref{thm:thm3}.

\section{Conclusion}
In this paper, we have studied the DD-DO algorithm in MASs. It is the first time that both the asynchrony in communication and the inexactness in solving subproblems are considered in the dual decomposition algorithm. Due to the asynchronous communication or non-identical computation clocks, agents have to solve their subproblems with the previously stored information. Limited by computational accuracy, the solutions to subproblems are inexact.
We have proved that values of primal and dual objectives converge to some neighborhoods of the optimal values, the solution converges to some neighborhood of the optimal solution, and the violation of constraints vanishes, all in the convergence rate of $\mathcal{O} ( 1 / \sqrt{k} )$.
Our convergence results generalize and unify existing works of dual decomposition algorithms considering only asynchrony or inexactness. Numerical simulation verifies the convergence performance of the asynchronous and inexact DD-DO algorithm. 

It is expected that this work could provide useful insights and facilitate the implementations of dual decomposition algorithms in complicated realistic systems, which would inspire more applications in a wide broad of fields.

\bibliographystyle{IEEEtran}
\bibliography{mybib}

\begin{thebibliography}{10}
\providecommand{\url}[1]{#1}
\csname url@samestyle\endcsname
\providecommand{\newblock}{\relax}
\providecommand{\bibinfo}[2]{#2}
\providecommand{\BIBentrySTDinterwordspacing}{\spaceskip=0pt\relax}
\providecommand{\BIBentryALTinterwordstretchfactor}{4}
\providecommand{\BIBentryALTinterwordspacing}{\spaceskip=\fontdimen2\font plus
\BIBentryALTinterwordstretchfactor\fontdimen3\font minus
  \fontdimen4\font\relax}
\providecommand{\BIBforeignlanguage}[2]{{%
\expandafter\ifx\csname l@#1\endcsname\relax
\typeout{** WARNING: IEEEtran.bst: No hyphenation pattern has been}%
\typeout{** loaded for the language `#1'. Using the pattern for}%
\typeout{** the default language instead.}%
\else
\language=\csname l@#1\endcsname
\fi
#2}}
\providecommand{\BIBdecl}{\relax}
\BIBdecl

\bibitem{papachristodoulou2010delay}
A.~Papachristodoulou and A.~Jadbabaie, ``Delay robustness of nonlinear internet
  congestion control schemes,'' \emph{IEEE Transactions on Automatic Control},
  vol.~55, no.~6, pp. 1421--1427, 2010.

\bibitem{he2011cross}
S.~He, J.~Chen, D.~K. Yau, and Y.~Sun, ``Cross-layer optimization of correlated
  data gathering in wireless sensor networks,'' \emph{IEEE Transactions on
  Mobile Computing}, vol.~11, no.~11, pp. 1678--1691, 2011.

\bibitem{magnusson2017convergence}
S.~Magn{\'u}sson, C.~Enyioha, N.~Li, C.~Fischione, and V.~Tarokh, ``Convergence
  of limited communication gradient methods,'' \emph{IEEE Transactions on
  Automatic Control}, vol.~63, no.~5, pp. 1356--1371, 2017.

\bibitem{komodakis2007mrf}
N.~Komodakis, N.~Paragios, and G.~Tziritas, ``Mrf optimization via dual
  decomposition: Message-passing revisited,'' in \emph{2007 IEEE 11th
  International Conference on Computer Vision}.\hskip 1em plus 0.5em minus
  0.4em\relax IEEE, 2007, pp. 1--8.

\bibitem{strandmark2010parallel}
P.~Strandmark and F.~Kahl, ``Parallel and distributed graph cuts by dual
  decomposition,'' in \emph{2010 IEEE Computer Society Conference on Computer
  Vision and Pattern Recognition}.\hskip 1em plus 0.5em minus 0.4em\relax IEEE,
  2010, pp. 2085--2092.

\bibitem{alkano2017asynchronous}
D.~Alkano, J.~M. Scherpen, and Y.~Chorfi, ``Asynchronous distributed control of
  biogas supply and multienergy demand,'' \emph{IEEE Transactions on Automation
  Science and Engineering}, vol.~14, no.~2, pp. 558--572, 2017.

\bibitem{huang2018stochastic}
S.~Huang, Y.~Sun, and Q.~Wu, ``Stochastic economic dispatch with wind using
  versatile probability distribution and l-bfgs-b based dual decomposition,''
  \emph{IEEE Transactions on Power Systems}, vol.~33, no.~6, pp. 6254--6263,
  2018.

\bibitem{falsone2019decentralized}
A.~Falsone, K.~Margellos, and M.~Prandini, ``A decentralized approach to
  multi-agent milps: finite-time feasibility and performance guarantees,''
  \emph{Automatica}, vol. 103, pp. 141--150, 2019.

\bibitem{Chiang1}
M.~Chiang, S.~H. Low, A.~R. Calderbank, and J.~C. Doyle, ``Layering as
  optimization decomposition: A mathematical theory of network architectures,''
  \emph{Proceedings of the IEEE}, vol.~95, no.~1, pp. 255--312, 2007.

\bibitem{Chiang2}
D.~P. {Palomar} and {Mung Chiang}, ``A tutorial on decomposition methods for
  network utility maximization,'' \emph{IEEE Journal on Selected Areas in
  Communications}, vol.~24, no.~8, pp. 1439--1451, Aug 2006.

\bibitem{Chiang3}
------, ``Alternative distributed algorithms for network utility maximization:
  Framework and applications,'' \emph{IEEE Transactions on Automatic Control},
  vol.~52, no.~12, pp. 2254--2269, Dec 2007.

\bibitem{bertsekas2003convex}
D.~Bertsekas and A.~Nedic, \emph{Convex analysis and optimization}.\hskip 1em
  plus 0.5em minus 0.4em\relax Athena Scientific, 2003.

\bibitem{bertsekas1997nonlinear}
D.~P. Bertsekas, \emph{Nonlinear programming}.\hskip 1em plus 0.5em minus
  0.4em\relax Athena Scientific, second edition, 1999.

\bibitem{wang2020asynchronous}
Z.~Wang, L.~Chen, F.~Liu, P.~Yi, M.~Cao, S.~Deng, and S.~Mei, ``Asynchronous
  distributed power control of multi-microgrid systems,'' \emph{IEEE
  Transactions on Control of Network Systems}, 2020.

\bibitem{bolognani2014distributed}
S.~Bolognani, R.~Carli, G.~Cavraro, and S.~Zampieri, ``Distributed reactive
  power feedback control for voltage regulation and loss minimization,''
  \emph{IEEE Transactions on Automatic Control}, vol.~60, no.~4, pp. 966--981,
  2014.

\bibitem{Steven}
S.~H. {Low} and D.~E. {Lapsley}, ``Optimization flow control. i. basic
  algorithm and convergence,'' \emph{IEEE/ACM Transactions on Networking},
  vol.~7, no.~6, pp. 861--874, Dec 1999.

\bibitem{magnusson2020distributed}
S.~Magn{\'u}sson, G.~Qu, and N.~Li, ``Distributed optimal voltage control with
  asynchronous and delayed communication,'' \emph{IEEE Transactions on Smart
  Grid}, 2020.

\bibitem{lee2015convergence}
K.~Lee and R.~Bhattacharya, ``On the convergence analysis of asynchronous
  distributed quadratic programming via dual decomposition,'' \emph{arXiv
  preprint arXiv:1506.05485}, 2015.

\bibitem{notarnicola2017distributed}
I.~Notarnicola, R.~Carli, and G.~Notarstefano, ``Distributed partitioned
  big-data optimization via asynchronous dual decomposition,'' \emph{IEEE
  Transactions on Control of Network Systems}, vol.~5, no.~4, pp. 1910--1919,
  2017.

\bibitem{devolder2014first}
O.~Devolder, F.~Glineur, and Y.~Nesterov, ``First-order methods of smooth
  convex optimization with inexact oracle,'' \emph{Mathematical Programming},
  vol. 146, no. 1-2, pp. 37--75, 2014.

\bibitem{necoara2013rate}
I.~Necoara and V.~Nedelcu, ``Rate analysis of inexact dual first-order methods
  application to dual decomposition,'' \emph{IEEE Transactions on Automatic
  Control}, vol.~59, no.~5, pp. 1232--1243, 2013.

\bibitem{fazlyab2018distributed}
M.~Fazlyab, S.~Paternain, A.~Ribeiro, and V.~M. Preciado, ``Distributed smooth
  and strongly convex optimization with inexact dual methods,'' in \emph{2018
  Annual American Control Conference (ACC)}.\hskip 1em plus 0.5em minus
  0.4em\relax IEEE, 2018, pp. 3768--3773.

\bibitem{zhang2020augmented}
Y.~Zhang and M.~M. Zavlanos, ``Augmented lagrangian optimization under
  fixed-point arithmetic,'' \emph{Automatica}, vol. 122, p. 109218, 2020.

\bibitem{mehyar2004optimization}
M.~Mehyar, D.~Spanos, and S.~H. Low, ``Optimization flow control with
  estimation error,'' in \emph{IEEE INFOCOM 2004}, vol.~2.\hskip 1em plus 0.5em
  minus 0.4em\relax IEEE, 2004, pp. 984--992.

\bibitem{beck20141}
A.~Beck, A.~Nedi{\'c}, A.~Ozdaglar, and M.~Teboulle, ``An $ o (1/k) $ gradient
  method for network resource allocation problems,'' \emph{IEEE Transactions on
  Control of Network Systems}, vol.~1, no.~1, pp. 64--73, 2014.

\bibitem{magnusson2018communication}
S.~Magn{\'u}sson, C.~Enyioha, N.~Li, C.~Fischione, and V.~Tarokh,
  ``Communication complexity of dual decomposition methods for distributed
  resource allocation optimization,'' \emph{IEEE Journal of Selected Topics in
  Signal Processing}, vol.~12, no.~4, pp. 717--732, 2018.

\bibitem{boyd2004convex}
S.~Boyd and L.~Vandenberghe, \emph{Convex optimization}.\hskip 1em plus 0.5em
  minus 0.4em\relax Cambridge university press, 2004.

\bibitem{ruszczynski2006nonlinear}
A.~Ruszczynski, \emph{Nonlinear optimization}.\hskip 1em plus 0.5em minus
  0.4em\relax Princeton university press, 2006.

\bibitem{wang2020asynchronous2}
Z.~Wang, F.~Liu, Y.~Su, P.~Yang, and B.~Qin, ``Asynchronous distributed voltage
  control in active distribution networks,'' \emph{Automatica}, vol. 122, p.
  109269, 2020.

\bibitem{nedic2009distributed}
A.~Nedic and A.~Ozdaglar, ``Distributed subgradient methods for multi-agent
  optimization,'' \emph{IEEE Transactions on Automatic Control}, vol.~54,
  no.~1, pp. 48--61, 2009.

\bibitem{mateos2016distributed}
D.~Mateos-N{\'u}nez and J.~Cort{\'e}s, ``Distributed saddle-point subgradient
  algorithms with laplacian averaging,'' \emph{IEEE Transactions on Automatic
  Control}, vol.~62, no.~6, pp. 2720--2735, 2016.

\bibitem{nedic2009approximate}
A.~Nedi{\'c} and A.~Ozdaglar, ``Approximate primal solutions and rate analysis
  for dual subgradient methods,'' \emph{SIAM Journal on Optimization}, vol.~19,
  no.~4, pp. 1757--1780, 2009.

\bibitem{website:CPLEX}
``Cplex,'' \url{https://www.ibm.com/analytics/cplex-optimizer}, 2020.

\bibitem{website:YALMIP}
``Yalmip,'' \url{https://yalmip.github.io/}, 2020.

\end{thebibliography}

\end{document}